\documentclass[11pt, reqno]{amsart}
\usepackage{amsfonts, amsthm,amsmath, amssymb, amscd}
\allowdisplaybreaks \allowdisplaybreaks[2]
\usepackage{mathtools}
\usepackage{esint}
\usepackage{graphics}
\usepackage{indentfirst}
\usepackage{latexsym}
\usepackage[dvips]{epsfig}
\usepackage{mathrsfs}
\usepackage{hyperref}
\usepackage{color}
\date{}
\numberwithin{equation}{section}

\usepackage{graphics}
\usepackage{epsfig}
\newtheorem{theorem}{Theorem}[section]

\newtheorem{lemma}[theorem]{Lemma}
\newtheorem{proposition}[theorem]{Proposition}
\newtheorem{remark}[theorem]{Remark}

\def\bmu{\boldsymbol{u}}
\def\bmv{\boldsymbol{v}}
\def\bmd{\boldsymbol{d}}
\def\tdiv{\operatorname{div}}

\begin{document}

\title[]{Global well-posedness of radially symmetric strong solutions to two-dimensional compressible liquid crystal flows with large data and vacuum}

\author{Yu~Mei}
\address[Yu Mei]{School of Mathematics and Statistics,  Northwestern Polytechnical University,  Xi'an 710129,  China}
\email{\href{mailto:yu.mei@nwpu.edu.cn}{yu.mei@nwpu.edu.cn}}

\author{Sen Yang}
%\thanks{Corresponding Author}
\address[Sen Yang]{School of Mathematics and Statistics,  Northwestern Polytechnical University,  Xi'an 710129,  China}
\email{\href{mailto:senyang@mail.nwpu.edu.cn}{senyang@mail.nwpu.edu.cn}}

\begin{abstract}
We study the initial boundary value problem of the two-dimensional compressible nematic liquid crystal 
flow with the shear viscosity $\mu$ being a positive constant and bulk viscosity $\lambda$ being a power function of density with the power exponent $\beta$.  Under the condition $\beta>1$, we establish the global existence and large time behavior of the radially symmetric strong solutions to this Vaigant--Kazhikhov type model of simplified compressible Ericksen-Leslie system of Dirichlet boundary conditions for the velocity and Neumann boundary ones for the director with arbitrary large data and vacuum. This work improves the results of Zhong and Zhou (\textit{Math. Ann.} \textbf{390}, 2024; \textit{J. Math. Pures Appl.}, \textbf{212}, 2026) for general 2D domains by removing the geometric angel condition on the director and relaxing the constrain on $\beta$ from \(\beta > 4/3\) to \(\beta > 1\). The key ingredient is that the rigidity mechanism arising from the radial symmetry of director prevents concentration phenomena in the transported harmonic heat flow.
\end{abstract}
\maketitle
%\tableofcontents

%%%%%%%%%%%%%%%%%%%%%%%%%%%%%%%%%%%%%%%%%%%%%%%%%%%%%%%%%%%%%%%%%%%%%%%%%%%%%%%%%%%%%%%%%%%%%%%%%%
\newcommand{\lnorm}[3]{{\left\| #1 \right\|}_{L^{#2}(#3)}}
\newcommand{\hnorm}[3]{{\left\| #1 \right\|}_{H^{#2}(#3)}}
\newcommand{\slnorm}[2]{{\left\| #1 \right\|}_{L^{#2}}}
\newcommand{\sllnorm}[3]{{\left\| #1 \right\|}_{L^{#2}L^{#3}}}

\section{Introduction and main result}
Liquid crystals are states of matter exhibiting both the fluidity of liquids and the orientational order characteristic of crystalline solids. In the nematic phase, the molecules tend to align along a preferred direction, which is usually described by a unit vector field $\bmd \in \mathbb{S}^2$, called the director field. The continuum description of liquid crystals can be traced back to the static theories of Oseen and Frank  \cite{oseen1933, frank1958}. 
In 1960s,  Ericksen \cite{ericksen1961,ericksen1962} and Leslie \cite{leslie1968} developed a hydrodynamic theory incorporating the interaction between the macroscopic fluid motion and the microscopic orientation of the liquid-crystal molecules. The resulting Ericksen--Leslie theory has become a fundamental continuum model for nematic liquid-crystal flows. Since the full Ericksen--Leslie system contains the general Oseen--Frank elastic energy for the director field together with the highly nonlinear anisotropic Leslie stress, Lin and Liu \cite{lin1995}  introduced a simplified incompressible model, which adopts the one-constant approximation of the Oseen--Frank elastic energy, and retains the basic competition between viscous dissipation, transport of the director, and orientational elasticity. This simplified system has played a central role in the mathematical analysis of liquid crystal flows. 

In the present paper, we are concerned with the simplified Ericksen--Leslie system in the compressible, isentropic setting. In a bounded domain  $\Omega \subset \mathbb{R}^{2}$, it takes the form
\begin{equation}\label{CEL-sys}
	\begin{cases}
		{\rho}_{t} + \text{div}(\rho \bmu)=0 \\
		(\rho \bmu)_{t} + \text{div}(\rho \bmu \otimes \bmu) + \nabla P = \mu \Delta \bmu + \nabla((\mu + \lambda)\text{div}\bmu) - \text{div}(\nabla\bmd\odot\nabla\bmd-\frac{1}{2}|\nabla\bmd|^2\mathbb{I}_2) \\
		\bmd_{t} +\bmu \cdot \nabla \bmd =\Delta \bmd + |\nabla \bmd|^{2}\bmd,
	\end{cases}
\end{equation}
where $\rho=\rho(x, t),  \bmu=(u_1, u_2)(x, t)$ and $\bmd=(d_1, d_2, d_3)(x, t)\in\mathbb{S}^2$ represent the density,  velocity and direction field,  respectively.  The pressure $P=P(\rho)$ is given by
\begin{equation}\label{gamma-law}
	P(\rho) =\rho^{\gamma}, \quad \gamma>1.
\end{equation}
While, $\mu$ and $\lambda$ are the shear and bulk viscosity coefficients respectively.  $\nabla\bmd\odot\nabla\bmd$  stands for  a $2\times 2$ matrix whose $ij$-component is $\partial_i\bmd\cdot\partial_j\bmd$ with $i, j=1, 2$. 
%We are interested in the initial-boundary problem (IBVP) of the system \eqref{CEL-sys} being supplemented with the initial data
%\begin{equation}\label{IC}
%	(\rho,\rho\bmu,\bmd)(x, 0)=(\rho_{0}(x),  \rho_0\bmu_{0}(x),\bmd_0(x)),\quad x \in \Omega, 
%\end{equation}
%and boundary conditions
%\begin{equation}\label{BC}
%	\bmu(x, t)=0, \frac{\partial \bmd}{\partial \boldsymbol{n}}=0,\quad x \in \partial \Omega, 
%\end{equation}
%where $\boldsymbol{n}$ is the unit outward normal vector to $\partial\Omega$.
% The viscosity coefficients $\mu $ and $\lambda $ are assumed to satisfy 
%\begin{equation}\label{VK-coef}
%	\mu =\mathrm{const.}>0,\quad \lambda(\rho)=\rho^{\beta}, 
%\end{equation}
%where $\beta >1$.  

The simplified compressbile Ericksen–Leslie system  \eqref{CEL-sys} couples the compressible isentropic Navier--Stokes equations with a transported harmonic-map heat flow into $\mathbb{S}^2$. In fact, when $\bmd$ is a constant vector, it reduces to the compressible Navier-Stokes equations, while the harmonic heat flow in the absence of fluid motion. From a mathematical point of view, this coupling creates difficulties that are substantially more delicate than those appearing in either subsystem separately. The momentum equation in \eqref{CEL-sys}, beside the possible degeneracy near vacuum,  contains a quadratic nonlinear term of $\nabla\bmd$ arising from the elasticity. Moreover, the  transport of director leads to the lack of monotonic type inequalities for the classical harmonic-map heat flow so that concentration of the director gradient may appear. Consequently, many existing global results impose a hemisphere or angle condition on $\bmd_0$ in order to recover a rigidity property for the director field. The presence of boundaries introduces further obstacles, since the effective viscous flux and singular integral arguments available on the torus or in the whole space no longer apply directly.

For the fluid part of  the coupled system \eqref{CEL-sys}, that is compressible isentropic Navier--Stokes equations, there is a huge literature concerning  the well-posedness of strong and weak solutions. When both viscosity coefficients $\mu,\lambda$ are positive constants, Lions, in his pioneering works \cite{Lions1993, lions1998}, established the global existence of finite-energy weak solutions with large initial data under the condition $\gamma \geq 9/5$. Later,  Feireisl,  Novotný and Petzeltová \cite{novotny2004book} subsequently developed a refined compactness theory reducing the restriction to $\gamma \geq 3/2$, 
while Jiang and Zhang \cite{JiangZhang2001} proved global existence of spherically symmetric weak solutions for all $\gamma > 1$.  As to global classical or strong solutions, Matsumura and Nishida\cite{Matsumura1980} established global well-posedness in 
the Sobolev space $H^s$ for initial perturbations near nonvacuum equilibrium states. See also Hoff \cite{Hoff.1995a,Hoff.1995c} the studies for discontinuous initial data.
A  major breakthrough concerning vacuum was made by Huang,  Li and Xin \cite{HuangLiXin2012}, who obtained global classical solutions in three dimensions  for initial data with sufficiently small total energy but possibly large oscillations and vacuum. Later,  Li and Xin  studied the two-dimensional Cauchy problem in \cite{LiXin2019}, and obtained  global well-posedness and large time behaviors of strong solutions when the 
initial total energy is suitably small.

When the viscosity coefficients $\mu,\lambda$ depend on the density, the structure of compressible isentropic Navier--Stokes equations changes significantly. One important class is formed by viscosity coefficients satisfying the B-D entropy relation $\lambda(\rho)=2(\rho\mu'(\rho)-\mu(\rho))$ introduced in \cite{Bresc.D2006a}, which offers the extra key estimates $\mu'(\rho)\sqrt{\rho}$.  Utilizing this estimate and Mellet-Vasseur type one, global weak solutions were obtained by  Li and Xin \cite{LiXin2015}, and independently by  Vasseur 
and Yu \cite{VasseurChenYu2016}, which were subsequently extended to more general nonlinear viscosity laws by
Bresch, Vasseur and Yu \cite{BreschVasseurChenYu2022}. However, the regularities and uniquness of such global weak solutions generally remain open in multi-dimensional domains. We refer \cite{Cao.L.Z2022,Cao.L.Z2024,Wen.Z2025} for global regular solutions with large data  and vacuum in one-dimensional or spherically symmetric cases. In the case that $\mu=\alpha\rho^\delta, \lambda=\beta\rho^\delta$, Xin and Zhu \cite{Xin.Z2021} recently established global well-posedness of regular solutions with small density but possibly large velocities provided $\delta>2\gamma-1$. Very recently, Huang, Li and Zhang \cite{Huang.L.Z2026a}, which extended Yu's result\cite{Yu.2023}, obtained the global existence of strong solutions with sufficient large initial density provided $\delta>\max\{\frac{\gamma+1}{2},\gamma-1\}$. Another important class of density-dependent viscosity model is the two-dimensional Vaigant--Kazhikhov model with the shear viscosity $\mu$ being a positive constant and the bulk one $\lambda=\rho^\beta$ being a power function of density, 
which provides one of the few mechanisms known to produce global classical solutions with arbitrarily large data for multi-dimensional problems.  
Vaigant and Kazhikhov \cite{VaigantKazhikhov1995} initiated this theory by obtaining global existence of strong solutions for large non-vacuum initial data in two space dimensional torus under the constraint $\beta >3$. The essential feature of their argument is that the growth of the bulk viscosity with the density can be exploited, together with the effective viscous flux, to prevent unbounded growth of the density. This observation inspired the works \cite{JiuQWangYXinZP2013,Jiu.W.X2014a, JiuQWangYXinZP2018,HuangXDLiJ2016, HuangXDLiJ2022,FanLiLi2022,HuangSu2024} aimed at admitting vacuum, extending the result to two-dimensional Cauchy or initial-boundary value problems, and, especially, decreasing the lower bound imposed on $\beta$. In particular, the improvement of the bulk viscosity exponent to $\beta>\frac{4}{3}$ was achieved by Huang and Li\cite{HuangXDLiJ2016, HuangXDLiJ2022} for the periodic domain and Cauchy problems, and by Fan, Li and Li \cite{FanLiLi2022} for IBVP with Navier-slip boundary conditions for the velocity. More recently, Huang, Su, Yan and Yu \cite{HuangSu2024} discovered that radial symmetry permits a further improvement from $\beta>\frac{4}{3}$ to $\beta>1$ even for IBVP in a two-dimensional solid ball with the Dirichlet condition, which shows that the exponent $\beta>4/3$, although fundamental in the available general two-dimensional arguments, is not an intrinsic threshold once additional geometric structure is exploited. However, whether this radial symmetric assumption can be removed in the case $\beta>1$ is still a challenge open problem even for a two-dimensional periodic domain.

For the original coupled system \eqref{CEL-sys} with constant viscosity coefficients, global existence of finite-energy weak solutions in a two or three dimensional domain were obtained by Jiang, Jiang and Wang \cite{jiangFjiangS2013, jiangFjiangS2014} under suitable restrictions on the initial energy or the geometric angular condition on the initial director field $\bmd_0$. Later, Lin, Lai and Wang \cite{LinJLaiB2015} imposed the hemisphere condition $\bmd_0(x)\in\mathbb S^2_+$ and extended the result of two-dimensional Cauchy problem  in \cite{jiangFjiangS2014} to three-dimensional IBVP. The hemisphere or angle conditions used in these works provide a rigidity mechanism that prevents concentration phenomena in the transported harmonic-map equation and permits strong convergence of the director gradients. Nevertheless,  removing such geometric constraints remains open even for the incompressible model in 
$\mathbb{R}^3 $. As to global strong and classical solutions, when initial data are close to the equilibrium state
$(1,  0,  \bar{\bmd})$ with a constant vector $\bar{\bmd} \in \mathbb{S}^2$, 
Hu and Wu \cite{Hu_2013} established global well-posedness of three-dimensional Cauchy problem in critical Besov spaces. Gao, Tao and Yao \cite{gaotao2016} obtained global existence and decay of classical solutions in $H^s(s\geq 3)$. We also refer to \cite{shibata2015} for global strong solutions in $L^p$ framework. These perturbations around non-vacuum state prevent the appearance of vacuum. For initial data allowing vacuum, Li, Xu and Zhang  \cite{LiJXuZ2018}  obtained global classical solutions to three dimensional Cauchy problem for initial data with sufficiently small total energy but possibly large oscillations and vacuum. Global strong solutions under small-energy hypotheses were subsequently obtained for other dimensions and boundary conditions. We refer to \cite{WangT2016} for two-dimensional Cauchy problem, \cite{LiuYZhongX2021} and \cite{SunYZhongX2023} for IBVP with Navier slip boundary conditions for velocity and Neumann boundary condition for orientation field. However, the results in  \cite{LiJXuZ2018, WangT2016, LiuYZhongX2021,SunYZhongX2023} rely essentially on the smallness of energy. For large initial data, Zhong and Zhou \cite{Zhongxin2024} investigated the Vaigant--Kazhikhov model for two-dimensional compressible Ericksen--Leslie system, and established the global
existence and uniqueness of strong solutions to the Cauchy problem with vacuum, provided that
$\beta > \frac{4}{3}$ and the initial director satisfies a geometric angle condition.  More recently, this result was further extended by Zhong and Zhou \cite{Zhongxin2026}to IBVP in a two-dimensional bounded domain with Navier-slip boundary conditions for the velocity and Neumann boundary conditions for the director. We also refer to \cite{Li.M.Z2025} for a large data result of compressible Ericksen-Leslie system in the case that $\mu=\alpha\rho^\delta, \lambda=\beta\rho^\delta$.

It is worth mentioning that, in the result of Huang, Su, Yan and Yu \cite{HuangSu2024} for the compressible Navier--Stokes equations,  the radial symmetry can lower the restriction on the bulk-viscosity exponent to $\beta>1$.  It is therefore natural to ask whether radial symmetry can be exploited to obtain the sharper Vaigant--Kazhikhov range $\beta>1$ for simplified compressible Ericksen-Leslie system, compared with the general two-dimensional results obtained by Zhong and Zhou \cite{Zhongxin2024, Zhongxin2026} requiring $\beta>4/3$. 
This radially symmetric problem was firstly studied by Huang and Wang \cite{Wang.H2018c} in the constant viscosity coefficients case. It is shown that the possible singularity occurs around the center with either the concentration or vanishing of density
or the norm inflammation of gradient of orientation field. It is interesting that whether this blow up phenomena can occur in the Vagiant-Kazhikhov model case.
 On the other hand, in \cite{Zhongxin2024, Zhongxin2026}, the geometric angel condition on the initial director field plays an essential role for general 2D problems. It is also interesting that whether the radial symmetry can be exploited to remove this angle condition. The answer is by no means an immediate consequence of the results in \cite{HuangSu2024,Zhongxin2024, Zhongxin2026}. In fact, the Ericksen stress, which is quadratic in $\nabla d$, appears in the estimates used to control the effective viscous flux and the density. Conversely, estimates for the director equation involve the velocity and its derivatives, so that these estimates are coupled with each other. Once the geometric angle condition is removed, one can no longer invoke the usual rigidity mechanism to control the harmonic-map nonlinearity and decouple the estimates so that a different idea base on the radial structure is required. 

Motivated by these observations, we investigate the initial-boundary value problem in a disk $\Omega$ of radius $R$ centered at the origin for the Vaigant--Kazhikhov model of two-dimensional radially symmetric simplified compressible Ericksen--Leslie system \eqref{CEL-sys}-\eqref{gamma-law}  with the viscosities satisfying
\begin{equation}\label{VK-coef}
	\mu =\mathrm{const.}>0,\quad \lambda(\rho)=\rho^{\beta},
\end{equation}
and initial-boundary conditions
\begin{equation}\label{IC}
	(\rho,\rho\bmu,\bmd)(x, 0)=(\rho_{0}(x),  \rho_0\bmu_{0}(x),\bmd_0(x)),\quad x \in \Omega, 
\end{equation}
and 
\begin{equation}\label{BC}
	\bmu(x, t)=0, \frac{\partial \bmd}{\partial \boldsymbol{n}}=0,\quad x \in \partial \Omega, 
\end{equation}
where $\boldsymbol{n}$ is the unit outward normal vector to $\partial\Omega$.

Our goal is to establish global existence, uniqueness and large time behavior of radially symmetric strong solutions 
\begin{equation}\label{symmetric assumption}
	\rho(x, t)=\rho(r, t), ~~\bmu(x, t)=u(r, t)\frac{x}{r}, ~~\bmd(x, t)=\bmd(r, t), \text{  with  }|x|=r
\end{equation}
for arbitrarily large initial data allowing vacuum under the natural condition
	\begin{equation*}
	\beta >1, \gamma >1.
\end{equation*} 
Under the radially symmetric assumption, the problem \eqref{CEL-sys}-\eqref{BC} reduce to the following system
\begin{equation}\label{sCEL-sys}
	\begin{cases}
		{\rho}_{t} + (\rho u)_r + \frac{\rho u}{r}=0 \\
		(\rho u)_t + (\rho u^2 )_r +\frac{\rho u^2}{r} + (\rho^{\gamma})_r=[(2\mu+\rho^{\beta})(u_r+\frac{u}{r})]_r-\frac{1}{2}({|\bmd_r|^2})_r-\frac{1}{r}{|\bmd_r|}^2\\
		\bmd_{t} +u \bmd_r =\bmd_{rr}+ {|\bmd_r|}^2\bmd+\frac{1}{r}\bmd_r
	\end{cases}
\end{equation}
with the initial data 
\begin{equation}\label{IC2}
	\rho (r, 0)=\rho_{0}(r),~~\rho u(r, 0)=\rho_0u_{0}(r),~~\bmd(r, 0)=\bmd_0(r), \quad 0 \le r < R , 
\end{equation}
and boundary conditions 
\begin{equation}\label{BC2}
	u(0, t)=u(R, t)=0,~~\bmd_r(0, t)=\bmd_r(R, t)=0, \quad t>0.  
\end{equation}
Our main results are stated as follows.
\begin{theorem}\label{main-thm}
	Assume that 
	\begin{equation}\label{TH1. 1-1}
		\beta >1, \gamma >1, 
	\end{equation} 
	and the initial data $(\rho_0, \bmu_0, \bmd_0)$ satisfy that for some $q>2$, 
	\begin{equation}\label{TH1. 1-2}
		\rho_0 \in W^{1, q}(\Omega), \bmu_0 \in H^1(\Omega), \nabla \bmd_0 \in H^1(\Omega).
	\end{equation}
Then the system \eqref{sCEL-sys}-\eqref{BC2} has a unique global strong solution  satisfying
	\begin{equation}\label{TH1. 1-4}
		\begin{cases}
			\rho \in C([0, T];W^{1, q}(\Omega)), \rho_t\in L^\infty(0,T;L^2),\\
			(\bmu,\nabla\bmd)\in L^{\infty}(0, T;H^1(\Omega))\cap L^{\frac{q+1}{q}}(0,T;W^{2,q})\\
			(\sqrt{t}\bmu,\sqrt{t}\nabla\bmd)\in L^2(0,T; W^{2,q}),\\ (\sqrt{t}\nabla^2 \bmu, \sqrt{t}\nabla^3 \bmd) \in L^{\infty}(0, T;L^2(\Omega)),\\
			 \sqrt{t}\bmu_t\in L^2(0, T;H^1(\Omega)),\sqrt{t}\nabla\bmd_t \in L^{\infty}(0, T;L^2(\Omega))\cap L^2(0, T;H^1(\Omega)),\\
			\rho\bmu\in C([0,T];L^2), \sqrt{\rho}\bmu_t\in L^2(\Omega\times(0,T)) 
		\end{cases}
	\end{equation}
	for any $T>0$. Moreover, there is a positive constant $C $ depending only on $\mu, \beta, \gamma$, $\|\rho_{0}\|_{W^{1,q}}$, $\|\bmu_0\|_{H^1}$, and $\|\nabla\bmd_0\|_{H^1}$ such that 
	\begin{equation}\label{thn12res}
		\sup_{0\le t <\infty} \slnorm{\rho}{\infty} \le  C,
	\end{equation}
	and the following asymptotic behaviors of solution hold
	\begin{equation}\label{large-time behavior}
		\lim_{t \rightarrow{\infty}} (\slnorm{\rho-\bar{\rho}_0}{p}+\slnorm{\nabla \bmu}{p} +\slnorm{\nabla^2\bmd}{p})=0, 
	\end{equation}
	for any $p \in [1, +\infty)$, where
	\begin{equation*}
		\bar{\rho}_0=\frac{1}{|\Omega|}\int \rho_0 dx. 
	\end{equation*}
\end{theorem}
\begin{remark}
	Theorem \ref{main-thm} improves the general two-dimensional results due to Zhong and Zhou \cite{Zhongxin2024, Zhongxin2026} in the sense that we removed the geometric angel condition of $\bmd$ and relaxed the admissible range of $\beta$ to $\beta>1$ under the radially symmetric assumption. Moreover, the time-independent bound of density and large time behavior are the firstly obtained for the Vaigiant-Kazhikhov model of compressible liquid crystal flows.
\end{remark}
\begin{remark}
	Theorem \ref{main-thm} generalizes the result of Huang, Su, Yan and Yu \cite{HuangSu2024} for compressible Navier-Stokes equations to simplified compressible Ericksen-Leslie system, and removed the restrictions $\gamma\geq\beta>\max\left\{1,\frac{\gamma+2}{4}\right\}$ on the uniform bound of density and large time behaviors. 
\end{remark}
\begin{remark}
	Theorem \ref{main-thm} excludes the possible center singularity for constant viscosity coefficients model studied in Huang and Wang \cite{Wang.H2018c}.
\end{remark}
\begin{remark}
	In Theorem \ref{main-thm}, we do not need any compatibility conditions on initial data. It is easy to  check that, similar to \cite{Zhongxin2024, Zhongxin2026}, if the initial data $(\rho_0,\bmu_0,\bmd_0)$ satisfy for some $q>2$ 
	$$\rho_0\in W^{2,q}(\Omega), \bmu_0 \in H^2(\Omega), \nabla \bmd_0 \in H^2(\Omega),$$
	and the 
	the following compatibility condition 
		\begin{equation*}
				-\mu \Delta \bmu_0-\nabla((\lambda + \mu) \mathrm{div} \bmu_0) +\nabla P_0 + \mathrm{div}(\nabla\bmd_0\odot\nabla\bmd_0-\frac{1}{2}|\nabla\bmd_0|^2\mathbb{I}_2) = \sqrt{\rho_0}g  
			\end{equation*}
		for some $g \in L^2$, then the strong solution can be improved to a classical one for any time.
\end{remark}
%The main ideas used in the proof of Theorem \ref{main-thm} are as follows. Similar to \cite{HuangSu2024,Zhongxin2024,Zhongxin2026}, the key issue is to obtain the a priori upper bound of the density. 

\section{Preliminaries}
In this section,  we recall some known facts and inequalities which will be used later. 

First, we can slightly modify the argument in \cite{HuangTWangC2012, Zhongxin2026} to obtain the following local existence of strong solutions to \eqref{CEL-sys}-\eqref{BC}. 
\begin{lemma}\label{lwp}
	Assume that $\beta \geq 1,\gamma > 1$ and  $(\rho_0, \bmu_0, \bmd_0)$ satisfy 
	$$\rho_0\geq\delta, \rho_0\in H^2, \bmu_0\in H^2,\bmd_0\in H^2, \bmu|_{\partial\Omega}=\frac{\partial\bmd}{\partial\boldsymbol{n}}|_{\partial \Omega}=0.$$
	Then there are a small time $T > 0$ and a constant $C_0 > 0$ depending only on $ \Omega, \gamma, \beta, \mu, \lambda$ and initial data such that there exists a unique strong solution $ (\rho, \bmu, \bmd) $ 
	to the problem \eqref{CEL-sys}-\eqref{BC} in $ \Omega \times (0, T) $ satisfying \eqref{TH1. 1-4} and away from vacuum.
\end{lemma}
Next, the following Gagliardo-Nirenberg and Sobolev inequalities (see \cite{Nirenberg1959}) will be used frequently. 
\begin{lemma}
	Let $\Omega$ be a bounded Lipschitz domain in $\mathbb{R}^2$. Then,  for any $ q>2$, there are positive constants $C_1, C_2$ depending on $q$ and $\Omega$, such that for any functions $\bmu,\bmv:\Omega\rightarrow\mathbb{R}^2$ with $\bmu\in H^1(\Omega)$ and $\bmv\in W^{1,q}(\Omega)$,
	\begin{equation}\label{lemma2-3}
		\|\bmu\|_{L^q(\Omega)}\leq C_1\|\bmu\|_{L^2(\Omega)}^\frac{2}{q}\|\nabla\bmu\|_{L^2(\Omega)}^{1-\frac{2}{q}}+C_2\|\bmu\|_{L^2(\Omega)}
	\end{equation}
	and
		\begin{equation}\label{lemma2-4}
		{\Vert \bmv \Vert}_{L^\infty(\Omega)} \le C_1{\Vert \bmv \Vert}_{L^2(\Omega)}^{\frac{q-2}{2q-2}}{\Vert \nabla \bmv \Vert}_{L^q(\Omega)}^{\frac{q}{2q-2}}+C_2\|\bmv\|_{L^2(\Omega)}. 
	\end{equation}
	Furthermore, if $ \bmu\cdot\boldsymbol{n}|_{\partial \Omega}=\bmv\cdot\boldsymbol{n}|_{\partial \Omega}=0\text{ or } \int_{\Omega} u_i dx=\int_\Omega v_i=0$ with $i=1,2$, then $C_2=0$.  
\end{lemma}

The following Brezis-Wainger inequalities cound be found in \cite{Brezis-Wainger-1980}. 
\begin{lemma} \label{Brezis-Wainger}
	For $p>2$,there exist a positive constant $C$ depending only on $p$ and $\Omega \subset \mathbb{R}^2$, such that every function $f \in W^{1,p}(\Omega)$ satisfies 
	\begin{equation*}
		\slnorm{f}{\infty} \leq C\slnorm{\nabla f}{2}\log^{1/2} (e+ \slnorm{\nabla f}{p}) +C\slnorm{f}{2}+C.
	\end{equation*}
\end{lemma}
The following Sobolev inequalities will be used frequently.
\begin{lemma}
	Let $\Omega$ be a bounded Lipschitz domain in $\mathbb{R}^2$. Then,  for any $ q>2$, there is a positive constant $C_q$ depending on $q$ and $\Omega$, such that for any functions $\bmu:\Omega\rightarrow\mathbb{R}^2$ with $ \bmu\cdot\boldsymbol{n}|_{\partial \Omega}=0\text{ or } \int_{\Omega} u_i dx=0, i=1,2$, and $\nabla\bmu\in L^{\frac{2q}{q+2}}(\Omega)$, we have
	\begin{equation}\label{lemma2-1}
		{\Vert \bmu \Vert}_{L^q(\Omega)} \le C_q{\Vert \nabla \bmu \Vert}_{L^{\frac{2q}{q+2}}(\Omega)}  . 
	\end{equation} 
\end{lemma}
The following Poincar\'{e} type inequality can be found in \cite{Feireisl2004}.
\begin{lemma}
	Let $\Omega$ be a bounded Lipschitz domain. If $\bmv\in H^1(\Omega)$ and $\rho$ is a non-negative function such that
	$$0<c\leq \int_\Omega\rho dx,\quad\int_\Omega\rho^\gamma\leq E_0$$
	with $\gamma>1$. Then, there exists a constant $C$ depending only on $\Omega, c, E_0$ and $\gamma$ such that
	\begin{equation}\label{Poincare-in}
		\|\bmv\|_{L^2(\Omega)}^2\leq C(\|\sqrt{\rho}\bmv\|_{L^2(\Omega)}^2+\|\nabla\bmv\|_{L^2(\Omega)}^2).
	\end{equation}
\end{lemma}
The following Beale-Kato-Majda type inequality (see \cite{HuangXLiJ2011}) will play an important role in the estimates of $\|\nabla\bmu\|_{L^\infty}$ and $\|\nabla\rho\|_{L^q}$ with $q>2$. 
\begin{lemma}
	Let $\Omega$ be a bounded Lipschitz domain in $\mathbb{R}^2$. Then, for any $ 2<q<\infty $, there is a constant $C$ depending on $q,\Omega$, such that the following estimate holds for all $\nabla \bmu \in W^{1, q}(\Omega),  $
	\begin{equation}\label{lemma2-5}
		{\Vert \nabla \bmu \Vert}_{L^\infty(\Omega)} \le C(\|\text{\rm div}\bmu \|_{L^\infty(\Omega)}+{\|\text{\rm curl} \bmu \|}_{L^\infty(\Omega)})\log(e+{\Vert \nabla^2 \bmu \Vert}_{L^q(\Omega)})+C{\Vert \nabla \bmu \Vert}_{L^2(\Omega)}+C. 
	\end{equation}
\end{lemma}
The following fact found in \cite{HuangSu2024} for the radially symmetric function $\bmu$ satisfying \eqref{symmetric assumption} and \eqref{BC2} will be used.
\begin{lemma}
	Assume that $\bmu(x, t)$ and $u(r, t)$ satisfy \eqref{symmetric assumption} and \eqref{BC2},  then it hold that 
	\begin{equation}\label{lemma2-2}
		{\Vert u \Vert}_{\infty} \le  {\Vert \text{\rm div} \bmu \Vert}_2 \le  {\Vert \nabla \bmu \Vert}_2. 
	\end{equation}
\end{lemma}
%\begin{proof}
%	For any $ r \in (0, R)$,  direct calculation leads to 
%	\begin{align*}
%		{|u(r)|^2} &= 2 \int_0^r u \partial_r udr \le 2 \int_0^R |\partial_ru||\frac{u}{r}|rdr \\
%		&\le \int_0^R {|\partial_ru|}^2rdr+ \int_0^R {|\frac{u}{r}|}^2rdr =\int_0^R{(\partial_ru+\frac{u}{r})}^2rdr \\
%		&= {\Vert div\bmu \Vert}_2^2.  
%	\end{align*}
%	which implies that ${\Vert u \Vert}_{\infty} \le {\Vert div\bmu \Vert}_2$. 
%\end{proof}

%\begin{lemma}
%	For any smooth function $f$, we have 
%	\begin{equation}\label{lemma2-6}
%		\int_{\Omega} \rho \frac{D}{Dt} f dx =\frac{d}{dt} \int_{\Omega} \rho fdx, 
%	\end{equation}
%	where $\rho$ is the density satisfying $(1, 1)_1$ and $\frac{D}{Dt}f$ denotes the material derivative of a function $f, i. e. $
%	\begin{equation}\label{lemma2-7}
%		\frac{D}{Dt} f=f_t+\bmu \cdot \nabla f. 
%	\end{equation}
%
%	\begin{proof}
%		Due to the function $f$ is smooth and $(1, 1)_1$, we have 
%		\begin{equation}
%			\begin{aligned}
%				\int_{\Omega} \rho \frac{D}{Dt} f dx &= \int_{\Omega}  \rho (f_t+\bmu \cdot \nabla f) dx\\
%				&=\int_{\Omega}  \rho f_t dx -\int_{\Omega} \tdiv (\rho \bmu) f dx=\int_{\Omega}  \rho f_t dx +\int_{\Omega} \rho_t f dx
%				=\frac{d}{dt} \int_{\Omega} \rho fdx. 
%			\end{aligned}
%		\end{equation}
%	\end{proof}
%\end{lemma}
We also need the following $L^p $ estimates for elliptic system with \textit{Neumann} boundary condition, which will be found in Chapter 7 in \cite{Trudinger1977}. 
\begin{lemma}
	Let $\Omega$ be a bounded domain with $C^{k+2}$ boundary.  Given $f \in W^{k,p}(\Omega)$,  if $v$ is the solution to following \textit{Neumann} boundary problem of the elliptic system equation
	\begin{equation}\label{elliptic_sys2}
		\begin{cases}
			\Delta v=f,  \enspace in \enspace \Omega\\
			\frac{\partial v}{\partial n}=0, \enspace on \enspace \partial\Omega
		\end{cases}
	\end{equation}
	satisfies the following elliptic estimates
	\begin{equation}\label{elliptic_est2}
		\begin{aligned}
			\|\nabla v\|_{W^{k+1,p}(\Omega)} \leq C\| f \|_{W^{k,p}(\Omega)}. 
		\end{aligned}
	\end{equation}
\end{lemma}

The following \textit{Zlotnik} inequality \cite{Zlotnik2000} will be used to get the time independent  upper bound of the density. 
\begin{lemma}\label{Zlotnik}
	Let the function $y \in W^{1, 1}(0, T) $ satisfy $$ y^{\prime}(t)=g(y)+h^{\prime}(t) \enspace on \enspace [0, T], y(0)=y_0,  $$
	with $g \in C(R)$ and $h \in W^{1, 1}(0, T). $ If $g(\infty)=-\infty$ and 
	\begin{equation}\label{lemma2-8}
		h(t_2)-h(t_1) \le N_0 +N_1(t_2-t_1)
	\end{equation}
	for all $0 \le t_1 < t_2 \le T$ with some $N_0 \ge 0$ and $N_1 \ge 0$, then $$y(t)\le max(y_0, \zeta_0) +N_0 <\infty \enspace on \enspace [0, T],  $$
	where $\zeta$ is a constant such that $g(\zeta) \le -N_1$ for $\zeta \ge \zeta_0 $. 
\end{lemma}

We also need the integral type Gronwall's inequality below.
\begin{lemma}
     Let $I=[a, b]$ and $\alpha(t), \beta(t), u(t)$ be a real-valued continuous function defined on $I$. 
    Moreover, $\alpha(t)$ is  a nondecreasing function and $\beta(t)$ is a nonnegative function. 
    if 
    \begin{equation}
        \begin{aligned}
            v(t) \le \alpha(t)+\int_{a}^{t} \beta(s) v(s) ds, \enspace t \in I, 
        \end{aligned}
    \end{equation}
    then, 
    \begin{equation}\label{integration-type-gronwall}
        \begin{aligned}
            v(t) \le 2\alpha(t)e^{\int_{a}^{t} \beta(s) ds}. 
        \end{aligned}
    \end{equation}
    \begin{proof}
        Define a function 
        \[
        g(s)=e^{-\int_{a}^{s} \beta(r) dr}\int_{a}^{s}\beta(r)v(r)dr, s \in I. 
        \]
        Differentiating $g(s)$ gives
        \[
        g'(s) = (v(s)-\int_{a}^{s} \beta(r)v(r) dr)\beta(s)e^{-\int_{a}^{s} \beta(r) dr}, 
        \] 
        so we have that
        \[
        g'(s) \le  \alpha(s) \beta(s)e^{-\int_{a}^{s} \beta(r) dr}, 
        \]
        Due to $g(a)=0$, integrating above inequality on $[a, t]$, we have that
        \[
        g(t) \le \int_{a}^{t} \alpha(s) \beta(s)e^{-\int_{a}^{s} \beta(r) dr} ds. 
        \]
        Using the defination of $g(t)$, we have that
        \[
        \int_{a}^{t} \beta(s)v(s) ds = e^{\int_{a}^{t} \beta(r)dr } g(t) \le  \int_{a}^{t} \alpha(s) \beta(s)e^{\int_{s}^{t} \beta(r) dr} ds. 
        \]
        Because $\alpha(t)$ is a nondecreasing function, so we have 
        \begin{equation}
            \begin{aligned}
                v(t) &\le \alpha(t) + \int_{a}^{t} \alpha(s) \beta(s)e^{\int_{s}^{t} \beta(r) dr} ds\\
                & \le \alpha(t)+ \alpha(t) \int_{a}^{t} \beta(s)e^{\int_{s}^{t} \beta(r) dr} ds \le \alpha(t)+ \alpha(t) [ -e^{\int_{s}^{t} \beta(r)dr } ]_{s=a}^{s=t}\\
                & \le \alpha(t)+ \alpha(t) e^{\int_{a}^{t} \beta(r)dr } \le 2\alpha(t) e^{\int_{a}^{t} \beta(r)dr }. 
            \end{aligned}
        \end{equation}
    \end{proof}
\end{lemma}

\section{Time-independent upper bound of the density}
In this section and next one,  we will always assume that $(\rho, \bmu, \bmd)$ is the radially symmetric solutions  to \eqref{sCEL-sys}-\eqref{BC2} in $\Omega \times (0, T]$ with $T>0$ obtained by Lemma \ref{lwp}. 

We will derive the time-independent upper bound of the density whenever $\beta,\gamma$ satisfy \eqref{TH1. 1-1} in this section, and left the higher order estimates in next one. We begin with the following basic energy estimates. 
\begin{lemma}\label{basic energy estimate}There is a constant $C>0$ depending only on $\gamma$, $\Omega$, $\|\rho_0\|_{L^\infty}$, $\|\bmu_0\|_{L^2}$ and $\|\nabla\bmd_0\|_{L^2}$ such that
	\begin{equation}\label{B-E-estimate}
		\begin{aligned}
			&\sup_{0 \le t \le T} \int (\rho {|\bmu|}^2+\rho^{\gamma}+{|\nabla \bmd|}^2)dx \\
			&+\int_0^T \int (\mu {|\nabla \bmu|}^2+(2\mu +\lambda(\rho)){(\tdiv \bmu)}^2+ {|\Delta \bmd+{|\nabla \bmd|}^2\bmd|}^2) dxdt \le C.
		\end{aligned}
	\end{equation}
%\begin{equation}\label{d-est}
%	\int_0^T \int  {|\partial_{rr} \bmd|}^2 + {|\partial_{r}\bmd|}^4+\frac{|\partial_r\bmd|}{r^2}dxdt  \leq C,
%\end{equation}
%and
%\begin{equation}\label{d-est2}
%	\int_{0}^{T}\|\nabla^2\bmd\|_{L^2}^2 dx+\int_0^T \slnorm{\nabla \bmd}{\infty}^2 dx  \leq C.
%\end{equation}
\end{lemma}
\begin{proof}
	Multiplying $\eqref{CEL-sys}_2$ and $\eqref{CEL-sys}_3$ by $\bmu$ and $-(\Delta \bmd+{|\nabla \bmd|}^2\bmd)$,  respectively,  summing the resultant equations up and integrating over $\Omega$,  then it follows from integration by parts,  $\eqref{CEL-sys}_1$ and \eqref{BC} that 
	$$\frac{d}{dt} \int \frac{1}{2}\rho {|\bmu|}^2 + \frac{\rho^{\gamma}}{\gamma-1}+\frac{1}{2} {|\nabla \bmd|}^2 dx+\int(\mu {|\nabla \bmu|}^2+(\mu +\lambda(\rho)){(\tdiv \bmu)}^2+{|\Delta \bmd+{|\nabla \bmd|}^2\bmd|}^2)dx=0. $$
	Integrating the above equality over $[0, T]$ gives that 
	\begin{equation}\label{BEe-proof-1}
		\begin{aligned}
			&\sup_{0 \le t \le T} \int \frac{1}{2}\rho {|\bmu|}^2 + \frac{\rho^{\gamma}}{\gamma-1}+\frac{1}{2} {|\nabla \bmd|}^2 dx\\
			&+\int_0^T  \int(\mu {|\nabla \bmu|}^2+(\mu +\lambda(\rho)){(\tdiv \bmu)}^2+{|\Delta \bmd+{|\nabla \bmd|}^2\bmd|}^2)dxdt \le C. 
		\end{aligned}
	\end{equation}
For $\bmu(x, t)=u(r, t)\frac{x}{r}$,  $\nabla\times\bmu=0$,  we have 
\begin{equation}\label{ellipest_1}
	\begin{cases}
		\Delta \bmu=\nabla \tdiv\bmu, \\
		\bmu|_{\partial \Omega}=0, 
	\end{cases}
\end{equation}
which yields that
\begin{equation}\label{idd}
	\int{|\nabla \bmu|}^2dx=\int |\tdiv\bmu|^2dx.
\end{equation}
Thus,  \eqref{B-E-estimate} follows from \eqref{BEe-proof-1} and \eqref{idd}.

%	Then, combining \eqref{dinfty} and \eqref{BEe-proof-4}, we can obtain
%	\begin{equation}
%		\begin{aligned}
%			\int_0^T \slnorm{\nabla \bmd}{\infty}^2 dx \le C. 
%		\end{aligned}
%	\end{equation}
%	Therefore,  we complete the proof. 
\end{proof}
Utilizing the radially symmetric property of $\bmd$, we can obtain the following key time-independent  $L^2_{t,x}$ estimates of $\nabla^2\bmd$ and $L^2_tL^\infty_x$ of $\nabla\bmd$.
\begin{lemma}\label{d-disp-est}
	There is a constant $C>0$ depending only on $\gamma$, $\Omega$, $\|\rho_0\|_{L^\infty}$, $\|\bmu_0\|_{L^2}$ and $\|\nabla\bmd_0\|_{L^2}$ such that
	\begin{equation}\label{d-est}
		\int_0^T \int  {|\partial_{rr} \bmd|}^2 + {|\partial_{r}\bmd|}^4+\frac{|\partial_r\bmd|}{r^2}dxdt  \leq C
	\end{equation}
	and
	\begin{equation}\label{d-est2}
		\int_{0}^{T}\|\nabla^2\bmd\|_{L^2}^2 dx+\int_0^T \slnorm{\nabla \bmd}{\infty}^2 dx  \leq C.
	\end{equation}
\end{lemma}
\begin{proof}
	Under the assumption of radial symmetry \eqref{symmetric assumption},  we have
	\begin{equation}\label{syscond_d}
		\begin{aligned}
			\Delta d_j(x, t)&=\partial_{rr}d_j(r, t) +\frac{1}{r}\partial_r d_j(r, t),  j=1, 2, 3\\
			\nabla d_j(x, t) &=\partial_r d_j(r, t)\frac{x}{r}, j=1, 2, 3\\
			{|\nabla \bmd|}^2&=\sum_{i=1}^2\sum_{j=1}^3 {|\partial_id_j|}^2= \sum_{i=1}^2\sum_{j=1}^3  {|\partial_r d_j \frac{x_i}{r}|}^2=\sum_{j} {|\partial_r d_j|}^2={|\partial_r \bmd|}^2. 
		\end{aligned}
	\end{equation}
	As a consequence,  
	\begin{equation}\label{BEe-proof-2}
		\begin{aligned}
			{\Vert \Delta \bmd+{|\nabla \bmd |}^2\bmd \Vert}_2^2
			&=\int_{\Omega} {|\partial_{rr}\bmd|}^2  +\frac{1}{r^2}{|\partial_{r}\bmd|}^2+ {|\partial_{r}\bmd|}^4+\frac{2}{r}{\partial_{rr}\bmd} \cdot {\partial_{r}\bmd}+2\partial_{rr}\bmd\cdot\bmd |\partial_{r}\bmd|^2dx\\
			&=\int_{\Omega} {|\partial_{rr}\bmd|}^2  +\frac{1}{r^2}{|\partial_{r}\bmd|}^2- {|\partial_{r}\bmd|}^4dx. 
		\end{aligned}
	\end{equation}
	where we have used
	$$ \int_{\Omega} \frac{1}{r}{\partial_{rr}\bmd} \cdot {\partial_{r}\bmd}dx=2\pi \int_0^R \frac{1}{2} \partial_r|\partial_r\bmd|^2dr=\pi(|\partial_r\bmd|^2(R)-|\partial_r\bmd|^2(0))=0$$
	and
	$$ -\partial_{rr}\bmd \cdot \bmd={|\partial_r \bmd|}^2$$
	which follows from $|\bmd|=1$.  It follows from ${|\partial_r \bmd|}^4 \le {|\partial_{rr} \bmd|}^2$ and \eqref{BEe-proof-2} that
	$$ \int_{\Omega} \frac{1}{r^2}{|\partial_{r}\bmd|}^2 dx\le {\Vert \Delta \bmd+{|\nabla \bmd |}^2\bmd \Vert}_2^2. $$
	Using \eqref{BEe-proof-1},  we have
	\begin{equation}\label{BEe-proof-3}
		\int_0^T \int_{\Omega} \frac{1}{r^2}{|\partial_{r}\bmd|}^2dx \le C. 
	\end{equation}
	Note that
	$$\int_{\Omega} \frac{1}{r^2}{|\partial_{r}\bmd|}^2 dx=2\pi\int_0^R \frac{1}{r}{|\partial_{r}\bmd|}^2 dr \ge \frac{2\pi}{R} \int_0^R {|\partial_{r}\bmd|}^2 dr. $$
	Then,  we have
	\begin{equation}\label{BEe-proof-4}
		\int_0^T \int_0^R{|\partial_{r}\bmd|}^2 dr  \le C. 
	\end{equation}
	%	using \eqref{BEe-proof-1}: $$ \int_{\Omega} {|\partial_{r}\bmd|}^4 dx \le {\Vert |\partial_r\bmd|^2 \Vert}_{\infty} \int_{\Omega} {|\partial_{r}\bmd|}^2 dx \le C {\Vert |\partial_r\bmd|^2 \Vert}_{\infty}. $$
	Using the Sobolev embeding $W^{1, 1}(I) \hookrightarrow L^{\infty}(I), I=(0, R)$,  we have
	\begin{equation}\label{dinfty}
		\begin{aligned}
			{\Vert |\nabla \bmd(x, t)|^2 \Vert}_{\infty(\Omega)}&={\Vert |\partial_r \bmd(x, t)|^2 \Vert}_{\infty(\Omega)}={\Vert |\partial_r \bmd(r, t)|^2 \Vert}_{\infty(I)}\\
			&\le C(\int_0^R {|\partial_{r}\bmd|}^2 dr+2\int_0^R |\partial_{r}\bmd||\partial_{rr}\bmd|dr)\\
			&\le C(\int_0^R {|\partial_{r}\bmd|}^2 dr+(\int_0^R \frac{1}{r} |\partial_{r}\bmd|^2 dr)^{\frac{1}{2}}(\int_0^R r |\partial_{rr}\bmd|^2 dr)^{\frac{1}{2}} )\\
			&\le \epsilon \int_{\Omega} |\partial_{r r}\bmd|^2 dx +C(\epsilon)(\int_0^R {|\partial_{r}\bmd|}^2 dr+\int_{\Omega} \frac{1}{r^2} |\partial_{r}\bmd|^2 dx), 
		\end{aligned}
	\end{equation}
	Therefore,  by using \eqref{BEe-proof-1},  we have
	\begin{equation*}
		\begin{aligned}
			\int_{\Omega} {|\partial_{r}\bmd|}^4 dx &\le {\Vert |\partial_r\bmd|^2 \Vert}_{\infty} \int_{\Omega} {|\partial_{r}\bmd|}^2 dx \le C {\Vert |\partial_r\bmd|^2 \Vert}_{\infty}\\&\le C \epsilon\int_{\Omega} |\partial_{r r}\bmd|^2 dx +C(\epsilon)(\int_0^R {|\partial_{r}\bmd|}^2 dr+\int_{\Omega} \frac{1}{r^2} |\partial_{r}\bmd|^2 dx). 
		\end{aligned}
	\end{equation*}
	Choosing $\epsilon$ small  and using \eqref{BEe-proof-3},  \eqref{BEe-proof-4} yield that
	\begin{equation}\label{BEe-proof-5}
		\int_{0}^{T}\int_{\Omega} {|\partial_{r}\bmd|}^4 dxdt\le \frac{1}{2} \int_{0}^{T}\int_{\Omega} |\partial_{r r}\bmd|^2 dxdt +C,
	\end{equation}
	%	$$ \int_{\Omega} {|\partial_{r}\bmd|}^4 dx \le \frac{1}{2} \int_{\Omega} |\partial_{r r}\bmd|^2 dx +C(\int_0^R {|\partial_{r}\bmd|}^2 dr+\int_{\Omega} \frac{1}{r^2} |\partial_{r}\bmd|^2 dx). $$
	which combines with \eqref{BEe-proof-2} and \eqref{BEe-proof-1} implies that
	\begin{equation}\label{BEe-proof-6}
		\int_{0}^{T}\int_{\Omega} {|\partial_{rr} \bmd|}^2dxdt\leq 2\int_{0}^{T}{\Vert \Delta \bmd+{|\nabla \bmd |}^2\bmd \Vert}_2^2dt+C\leq C.
	\end{equation}
	Therefore, \eqref{d-est} is a consequence of \eqref{BEe-proof-3}, \eqref{BEe-proof-5} and \eqref{BEe-proof-6}.
	Furthermore, \eqref{d-est2} follows from \eqref{dinfty}, \eqref{BEe-proof-4}, \eqref{d-est} and the fact $|\nabla^2\bmd|\leq C\left(|\partial_{rr}\bmd|+|\partial_r\bmd|/r\right).$
\end{proof}
We can further utilize Lemma \ref{d-disp-est} to get the time-independent $L^\infty_tL^4_x$ estimate of $\nabla\bmd$.  
\begin{lemma}
	There exists a positive constant $C$ depending only $\gamma$, $\Omega$, $\|\rho_0\|_{L^\infty}$, $\|\bmu_0\|_{L^2}$ and $\|\nabla\bmd_0\|_{L^2}$  such that
	\begin{equation}\label{est_d_1}
		\sup_{0\le t \le T} \int_{\Omega} |\nabla \bmd|^4 dx+\int_0^T \int_{\Omega} |\nabla \bmd|^2 |\Delta \bmd|^2 dx dt+\int_0^T \int_{\Omega} |\nabla \bmd|^2|\nabla(|\nabla \bmd|)|^2dxdt \le C. 
	\end{equation}
\end{lemma}
\begin{proof}
	Applying $\nabla$ on $\eqref{CEL-sys}_3$,    we get that
	\begin{equation}\label{nabula_cel3}
		\nabla \bmd_t -\Delta \nabla \bmd=-\nabla (\bmu \cdot \nabla \bmd)+\nabla(|\nabla \bmd|^2 \bmd). 
	\end{equation}
	Multiplying above by $ 4 |\nabla \bmd|^2 \nabla \bmd$,  integrating by parts and using \eqref{BC} yield that
	\begin{equation}\label{est_d1_proof_1}
		\begin{aligned}
			\frac{d}{dt} \int |\nabla \bmd|^4 dx &+4 \int (|\nabla \bmd|^2|\nabla^2 \bmd|^2+2|\nabla \bmd|^2|\nabla(|\nabla \bmd|)|^2)dx \\
			&= 4 \int |\nabla \bmd|^2 \nabla \bmd \cdot [-\nabla(\bmu \cdot \nabla \bmd)+\nabla (|\nabla \bmd|^2 \bmd)]dx  \\
			&\le C\int |\nabla \bmd|^3|\nabla^2 \bmd||\bmu|dx+C \int|\nabla^2\bmd||\nabla \bmd|^4 dx \\
			&\le 2 \int |\nabla^2\bmd|^2|\nabla \bmd|^2 dx +C\int |\bmu|^2|\nabla \bmd|^4 dx+C\int |\nabla \bmd|^6 dx\\
			&\le 2 \int |\nabla^2\bmd|^2|\nabla \bmd|^2 dx +C(\|\bmu\|_{\infty}^2+\||\nabla\bmd|^2\|_{\infty})\int|\nabla \bmd|^4 dx. 
		\end{aligned}
	\end{equation}
	%		In view of \eqref{lemma2-2},  \eqref{B-E-estimate},  \eqref{dinfty},  \eqref{BEe-proof-3} and \eqref{d-est},  we have
	%		\begin{equation*}
		%			\int_{0}^{T}\|\bmu\|_{\infty}^2dt\leq \int_{0}^{T}\|\nabla\bmu\|_2^2dt\leq C
		%		\end{equation*}
	%	and
	%		\begin{align*}
		%			\int_0^T \lnorm{|\partial_r\bmd|^2}{\infty}{\Omega} dt &\le C\int_0^T \left(\int_{\Omega} |\partial_{r r}\bmd|^2 dx +\int_0^R {|\partial_{r}\bmd|}^2 dr+\int_{\Omega} \frac{1}{r^2} |\partial_{r}\bmd|^2 dx\right)dt\le C. 
		%		\end{align*}
	%		Plugging these into \eqref{est_d1_proof_1} gives that
	%		\begin{align*}
		%			\frac{d}{dt} \int |\nabla \bmd|^4 dx &+4 \int (|\nabla \bmd|^2|\nabla^2 \bmd|^2+2|\nabla \bmd|^2|\nabla(|\nabla \bmd|)|^2)dx \\
		%			&\le  2 \int |\nabla^2\bmd||\nabla \bmd|^2 dx +C(\lnorm{\bmu}{\infty}{\Omega}^2+\lnorm{|\partial_r\bmd|^2}{\infty}{\Omega})\int |\nabla \bmd|^4 dx \\
		%			&\le 2 \int |\nabla^2\bmd||\nabla \bmd|^2 dx +C(\lnorm{\nabla \bmu}{2}{\Omega}^2+\lnorm{|\partial_r\bmd|^2}{\infty}{\Omega})\int |\nabla \bmd|^4 dx. 
		%		\end{align*}
	Therefore,  we can obtain from applying Gronwall's inequality to \eqref{est_d1_proof_1} that
	\begin{equation}\label{est_d1_proof_2}
		\begin{aligned}
			\int_{\Omega} |\nabla \bmd|^4 dx &+2\int_0^t \int_{\Omega} |\nabla \bmd|^2 |\Delta \bmd|^2 dx dt+8\int_0^t \int_{\Omega} |\nabla \bmd|^2|\nabla(|\nabla \bmd|)|^2dxdt  \le C, 
		\end{aligned}
	\end{equation}
	which competes the proof. 
\end{proof}

Let $G$ be the effective flux as
 \begin{equation}\label{G-def}
 	G:=F+\bar{P}=(2\mu+\lambda)\tdiv \bmu-(P-\bar{P}),
 \end{equation}
 where $\bar{P}:= \frac{1}{|\Omega|}\int P dx$ is the integral average of $P$ on $\Omega$. 
 
 Define
 \begin{equation}\label{A-def}
 	A(t)=\left\|\frac{G(t)}{\sqrt{2\mu+\lambda(\rho(t))}}\right\|_{L^2(\Omega)}+\left\|\nabla^2 \bmd(t)\right\|_{L^2(\Omega)},
 \end{equation}
 and
 \begin{equation}
 B(t)=\left\|\sqrt{\rho(t)}\dot{\bmu}(t)\right\|_{L^2(\Omega)}+\left\|\nabla^3\bmd(t)\right\|_{L^2(\Omega)}. 
 \end{equation}
Then, we have the following estimates of $A(t)$ and $B(t)$ in terms of $$R_T:=1+ \sup_{0<t<T}\|\rho\|_{L^\infty},$$ which plays important role in getting the time-independent upper bound of the density. 
\begin{lemma}\label{log-est}
	For any $\alpha\in(0, 1)$,  there is a constant $C(\alpha)$ depending only on $\alpha, \mu, \beta, \gamma$, $\Omega$, $\|\rho_0\|_{L^\infty}$, $\|\bmu_0\|_{H^1}$ and $\|\nabla\bmd_0\|_{H^1}$ such that
	\begin{equation}\label{est_AB}
	\begin{aligned}
		\sup_{0\le t\le T} \ln(e+A^2(t)) \leq CR_T^{\alpha\beta+1}
	\end{aligned}
    \end{equation}
	and 
    \begin{equation}\label{est_AB_another}
	\begin{aligned}
		\sup_{0\le t\le T} (e+A^2(t)) +\int_0^T B(t)^2 dt\le CR_T^{\alpha\beta+1}(R_T+R_T^{\beta-\gamma-1}+R_T^{\gamma-2\beta+1})e^{CR_T^{\alpha\beta+1}}. 
	\end{aligned}
    \end{equation}
	\begin{proof}
		Using radially symmetric property,  we rewrite the momentum equation $\eqref{CEL-sys}_2$ as:
        \begin{equation}\label{eq_dotu}
		\begin{aligned}
			\rho \dot{\bmu}=\nabla G -\text{div}(\nabla\bmd\odot\nabla\bmd-\frac{1}{2}|\nabla\bmd|^2\mathbb{I}_2). 
		\end{aligned}
        \end{equation}
%    where $M(\bmd)=\nabla\bmd\odot\nabla\bmd-\frac{1}{2}|\nabla\bmd|^2\mathbb{I}_2. $
		Multiplying \eqref{eq_dotu}by $\dot{\bmu}$,  integrating the resultant equality over $\Omega$ and using integration by parts,  we obtain that
        \begin{equation}\label{estAB_proof_01}
                \int \rho |\dot{\bmu}|^2dx=-\int G \text{div} \dot{\bmu}dx+\int \nabla\bmd\odot\nabla\bmd:\nabla \dot{\bmu}dx-\frac{1}{2}\int|\nabla\bmd|^2\tdiv{\dot{\bmu}}dx.
        \end{equation}
        Noticing that 
        \begin{equation*}
            \begin{aligned}
                \text{div} \dot{\bmu} =\frac{D}{Dt} \text{div}\bmu +(\text{div} \bmu)^2 -2\partial_ru\frac{u}{r}, 
            \end{aligned}
        \end{equation*}
        one has
        \begin{equation*}
            \begin{aligned}
                -\int G\tdiv \dot{\bmu} dx &=-\int G\frac{D}{Dt} \tdiv \bmu dx -\int G(\tdiv \bmu)^2 dx +2 \int G\partial_ru\frac{u}{r} dx\\
                &=-\int G \frac{D}{Dt}\left(\frac{G+P-\bar{P}}{2\mu+\lambda}\right)dx-\int G\frac{G+P-\bar{P}}{2\mu+\lambda}\tdiv \bmu dx +2 \int G\partial_ru\frac{u}{r} dx\\
                &=-\frac{1}{2} \frac{d}{dt} \int \frac{G^2}{2\mu+\lambda}dx+\frac{1}{2}\int \frac{G^2}{2\mu+\lambda}\text{div}\bmu dx-\frac{1}{2} \int G^2 \frac{\rho \lambda'}{(2\mu + \lambda)^2} \tdiv \bmu dx\\
                &\quad+\gamma\int GP\frac{1}{2\mu+\lambda}\tdiv\bmu dx+\int G\partial_t\bar{P}\frac{1}{2\mu+\lambda}dx-\int G(P-\bar{P})\frac{\rho\lambda'}{(2\mu +\lambda)^2}\text{div}\bmu dx\\
                &\quad-\int \frac{G^2}{2\mu+\lambda}\tdiv\bmu dx-\int \frac{G(P-\bar{P})}{2\mu+\lambda}\tdiv\bmu dx+2 \int G\partial_ru\frac{u}{r} dx\\
                &\leq -\frac{1}{2} \frac{d}{dt} \int \frac{G^2}{2\mu+\lambda}dx+C\int \frac{G^2}{2\mu+\lambda}|\text{div}\bmu| dx+C\int \frac{|G|(|P|+|\bar{P}|)}{2\mu+\lambda}|\text{div}\bmu| dx\\
                &\quad+C\left|\int P\tdiv\bmu dx\cdot\int \frac{G}{2\mu+\lambda}dx\right|,  
            \end{aligned}
        \end{equation*}
    where we have used
   \begin{align*}
   	\left|2 \int G\partial_ru\frac{u}{r} dx\right|&\leq \left|\int G\left(\partial_r u+\frac{u}{r}\right)^2dx\right|=\left|\int G|\tdiv\bmu|^2dx\right|=\left|\int G\frac{G+P-\bar{P}}{2\mu+\lambda}\tdiv\bmu dx\right|\\
   	&\leq \int \frac{G^2}{2\mu+\lambda}|\text{div}\bmu| dx+\int \frac{|G||P-\bar{P}|}{2\mu+\lambda}|\text{div}\bmu| dx 
   \end{align*}
and
\begin{equation*}
	\partial_t\bar{P}=\frac{1}{|\Omega|}\int\partial_t Pdx=-\frac{\gamma-1}{|\Omega|}\int P\tdiv\bmu dx. 
\end{equation*}
%        Observing that
%        \begin{equation*}
%            \begin{aligned}
%                \int F\frac{D}{Dt}\tdiv \bmu dx &= \int F \frac{D}{Dt}(\frac{F+P}{2\mu+\lambda})dx\\
%                &=\frac{1}{2} \frac{d}{dt} \int \frac{F^2}{2\mu+\lambda}dx-\frac{1}{2}\int \frac{F^2}{2\mu+\lambda}\text{div}\bmu dx\\
%                &\quad \enspace -\frac{1}{2} \int F^2 \frac{\rho \lambda'}{(2\mu + \lambda)^2} \tdiv \bmu dx - \int F\rho(\frac{P}{2\mu +\lambda})' \text{div}\bmu dx. 
%            \end{aligned}
%        \end{equation*}
%        Therefore, 
%        \begin{equation*}
%            \begin{aligned}
%                \int \rho |\dot{\bmu}|^2 &=-\int  F \tdiv \dot{\bmu}-\int \text{div}(M(\bmd)) \cdot \dot{\bmu}\\
%                &=- \frac{d}{dt} \int \frac{F^2}{2\mu+\lambda}dx+\int \frac{F^2}{2\mu+\lambda} dx\\
%                &\quad \enspace - \int F^2 \frac{\rho \lambda'}{(2\mu + \lambda)^2} \tdiv \bmu dx +2 \int F\rho(\frac{P}{2\mu +\lambda})' dx \\
%                &\quad \enspace -2\int F|\tdiv \bmu|^2 dx +4 \int F\partial_ru\frac{u}{r} dx-2\int \text{div}(M(d)) \cdot \dot{\bmu}. 
%            \end{aligned}
%        \end{equation*}
%		Then we calculate that
Using integration by parts gives that
		\begin{align*}
			&\int \nabla\bmd\odot\nabla\bmd:\nabla \dot{\bmu}\, dx\\
			&=\frac{d}{dt} \int (\nabla \bmd \odot \nabla \bmd) :\nabla \bmu dx-\int (\nabla \bmd_t \odot \nabla \bmd+\nabla \bmd \odot \nabla \bmd_t) :\nabla \bmu dx -\int \nabla \bmd \odot \nabla \bmd:\nabla (\bmu \cdot \nabla \bmu) dx\\
			&=\frac{d}{dt} \int (\nabla \bmd \odot \nabla \bmd):\nabla \bmu dx -\int [(\Delta \nabla \bmd -\nabla(\bmu\cdot \nabla \bmd) + \nabla (|\nabla \bmd|^2\bmd) ) \odot \nabla \bmd]:\nabla \bmu dx \\
			&\quad-\int [ \nabla \bmd\odot(\Delta \nabla \bmd -\nabla(\bmu\cdot \nabla \bmd) + \nabla (|\nabla \bmd|^2\bmd) ) ]:\nabla \bmu dx\\
			&\quad-\int \nabla \bmd \odot \nabla \bmd: (\bmu \cdot \nabla) \nabla\bmu dx -\int \nabla \bmd \odot \nabla \bmd:\nabla \bmu \cdot \nabla \bmu dx\\
			&=\frac{d}{dt} \int (\nabla \bmd \odot \nabla \bmd):\nabla \bmu dx -\int [(\Delta \nabla \bmd -\nabla\bmu\cdot \nabla \bmd +  |\nabla \bmd|^2\nabla\bmd\odot \nabla \bmd]:\nabla \bmu dx \\
			&\quad -\int [ \nabla \bmd\odot(\Delta \nabla \bmd -\nabla\bmu\cdot \nabla \bmd + \nabla \bmd|\nabla \bmd|^2) ]:\nabla \bmu dx\\
			&\quad+\int\tdiv\bmu \nabla \bmd \odot \nabla \bmd: \nabla\bmu dx-\int \nabla \bmd \odot \nabla \bmd:\nabla \bmu \cdot \nabla \bmu dx\\
			&\leq \frac{d}{dt} \int (\nabla \bmd \odot \nabla \bmd):\nabla \bmu dx+C\int(|\nabla\Delta \bmd|+|\nabla\bmu||\nabla\bmd|+|\nabla \bmd|^3)|\nabla\bmd||\nabla\bmu|dx,
		\end{align*}
	        where we have used 
	        \begin{align*}
	        	&-\int ((\bmu\cdot\nabla)\nabla\bmd \odot \nabla \bmd+\nabla \bmd \odot (\bmu\cdot\nabla)\nabla\bmd) :\nabla \bmu dx-\int \nabla \bmd \odot \nabla \bmd: (\bmu \cdot \nabla) \nabla\bmu dx\\
%	        	&=-\int (u_k \partial_k \partial_i \bmd \cdot \partial_j \bmd+u_k \partial_k \partial_j \bmd \cdot \partial_i \bmd) \partial_j u_i dx+\int \partial_k(\partial_i\bmd\cdot\partial_j\bmd)u_k\partial_j u_i dx\\
%	        	&\quad+\int\tdiv\bmu \nabla \bmd \odot \nabla \bmd: \nabla\bmu dx\\
	        	&=\int\tdiv\bmu \nabla \bmd \odot \nabla \bmd: \nabla\bmu dx.
	        \end{align*}
    Moreover, 
	\begin{align*}
		-\frac{1}{2}\int|\nabla\bmd|^2\tdiv{\dot{\bmu}}dx&=-\frac{1}{2}\frac{d}{dt}\int|\nabla\bmd|^2\tdiv\bmu dx+\int\nabla\bmd_t:\nabla\bmd\tdiv\bmu dx-\frac{1}{2}\int|\nabla\bmd|^2\tdiv(\bmu\cdot\nabla\bmu)dx\\
		&=-\frac{1}{2}\frac{d}{dt}\int|\nabla\bmd|^2\tdiv\bmu dx+\int(\Delta \nabla \bmd -\nabla \bmu\cdot \nabla \bmd +|\nabla \bmd|^2\nabla\bmd):\nabla\bmd\tdiv\bmu dx\\
		&\quad-\frac{1}{2}\int|\nabla\bmd|^2(\tdiv\bmu)^2 dx-\frac{1}{2}\int|\nabla\bmd|^2\partial_iu_j\partial_j u_idx\\
		 &\le -\frac{1}{2}\frac{d}{dt}\int|\nabla\bmd|^2\tdiv\bmu dx +C\int(|\nabla\Delta\bmd|+|\nabla \bmd|^3)|\nabla\bmd||\nabla \bmu|dx\\
		 &\quad+C\int|\nabla\bmd|^2|\nabla\bmu|^2dx. 
	\end{align*}
        Combining all above,  we have
        \begin{equation}\label{estAB_proof_1}
		\begin{aligned}
			&\frac{1}{2}\frac{d}{dt} \left\|\frac{G}{\sqrt{2\mu+\lambda(\rho)}}\right\|_{L^2}^2 +\lnorm{\sqrt{\rho}\dot{\bmu}}{2}{\Omega}^2\\
			&\le \frac{d}{dt}\left(\int (\nabla \bmd\odot \nabla \bmd): \nabla \bmu dx-\frac{1}{2}\int |\nabla \bmd|^2\tdiv\bmu dx\right)+C\int \frac{G^2}{2\mu+\lambda}|\text{div}\bmu| dx\\
			&\quad+C\int \frac{|G|(|P|+|\bar{P}|)}{2\mu+\lambda}|\text{div}\bmu| dx+C\left|\int P\tdiv\bmu dx\cdot\int \frac{G}{2\mu+\lambda}dx\right|\\
            &\quad+C\int|\nabla\Delta \bmd||\nabla\bmd||\nabla\bmu|dx+C\int|\nabla \bmd|^4|\nabla\bmu|dx+C\int|\nabla\bmd|^2|\nabla\bmu|^2dx\\
            &=:\frac{d}{dt}I_0+\sum_{i=1}^6 I_i. 
		\end{aligned}
        \end{equation}
    Each $I_i$ can be estimated as follows. 
     For any $0 <\alpha<1$,  it follows from H\"{o}lder's inequality and \eqref{lemma2-3} that
    \begin{equation}\label{I1-est}
    	\begin{aligned}
    		I_1&\le  \slnorm{\frac{G^2}{2\mu +\lambda}}{2} \slnorm{\tdiv\bmu}{2}
    		\\
    		&\le \slnorm{\frac{G}{\sqrt{2\mu +\lambda}}}{2}^{1-\alpha} \slnorm{G}{\frac{2(1+\alpha)}{\alpha}}^{1+\alpha} \slnorm{\sqrt{2\mu +\lambda}\tdiv\bmu}{2}\\
    		&\le \slnorm{\frac{G}{\sqrt{2\mu +\lambda}}}{2}^{1-\alpha} \slnorm{G}{2}^{\alpha}\|G\|_{H^1}\slnorm{\sqrt{2\mu +\lambda}\tdiv\bmu}{2}\\
    		&\le R_T^{\frac{\alpha \beta}{2}} \slnorm{\frac{G}{\sqrt{2\mu +\lambda}}}{2}R_T^{\frac{1}{2}} (\slnorm{\sqrt{\rho }\dot{\bmu}}{2} +\slnorm{\nabla^3 \bmd}{2}+\slnorm{\nabla^2 \bmd}{2})\slnorm{\sqrt{2\mu +\lambda}\tdiv \bmu}{2}\\
    		&\quad+R_T^{\frac{\alpha \beta}{2}} \slnorm{\frac{G}{\sqrt{2\mu +\lambda}}}{2}(1+R_T^{\frac{\beta-\gamma}{2}})\|\sqrt{2\mu+\lambda}\tdiv\bmu\|_{L^2}^2\\
    		&\le \frac{1}{4} \slnorm{\sqrt{\rho }\dot{\bmu}}{2}^2+\frac{1}{16}\slnorm{\nabla^3 \bmd}{2}^2+C\slnorm{\nabla^2 \bmd}{2}\\
    		&\quad+C(\alpha) R_T^{\alpha \beta +1} \|\frac{G}{\sqrt{2\mu +\lambda}}\|_{L^2}^2 \slnorm{\sqrt{2\mu +\lambda}\tdiv \bmu}{2}^2\\
    		&\quad+C(R_T+R_T^{\beta-\gamma-1})\slnorm{\sqrt{2\mu +\lambda}\tdiv \bmu}{2}^2. 
    	\end{aligned}
    \end{equation}
    where we have used the following facts
    \begin{equation*}
    	\begin{aligned}
    		&\slnorm{G}{\frac{2(1+\alpha)}{\alpha}}\leq C\|G\|_{L^2}^\frac{\alpha}{1+\alpha}\|G\|_{H^1}^\frac{1}{1+\alpha}, \\
    		&\|G\|_{L^2}\leq R_T^\frac{\beta}{2}\left\|\frac{G}{\sqrt{2\mu+\lambda}}\right\|_{L^2}, \\
    		&|\bar{G}|\leq (1+R_T^{\frac{\beta-\gamma}{2}})\|\sqrt{2\mu+\lambda}\tdiv\bmu\|_{L^2}, \\
    		&\slnorm{\nabla G}{2}
    		\le R_T^\frac{1}{2} \slnorm{\sqrt{\rho} \dot{\bmu}}{2} +\slnorm{\nabla \bmd}{4}\slnorm{\nabla^2 \bmd}{4}
    	\end{aligned}
    \end{equation*}
	and
	\begin{equation}\label{est_G_H1}
		\begin{aligned}
			\|G\|_{H^1}&\leq \|G-\bar{G}\|_{L^2}+|\bar{G}|+\|\nabla G\|_{L^2}\leq |\bar{G}|+C\|\nabla G\|_{L^2}\\
		&\le CR_T^\frac{1}{2}( \slnorm{\sqrt{\rho} \dot{\bmu}}{2} +\slnorm{\nabla^3 \bmd}{2} +\slnorm{\nabla^2 \bmd}{2})+C(1+R_T^{\frac{\beta-\gamma}{2}})\|\sqrt{2\mu+\lambda}\tdiv\bmu\|_{L^2}. 
		\end{aligned}
	\end{equation}

    Similarly,  we have the following estimate of $I_2$:
    \begin{equation}
    	\begin{aligned}
    		I_2 &\le \slnorm{G}{2+4\gamma/\beta} \slnorm{\frac{1+P}{(2\mu +\lambda)^\frac{3}{2}}}{2+\frac{\beta}{\gamma}} \slnorm{\sqrt{2\mu+\lambda} \tdiv\bmu}{2}\\
    		&\le \hnorm{G}{1}{\Omega} (1+R_T^{\frac{\gamma-2\beta}{2}}) \slnorm{\sqrt{2\mu+\lambda} \tdiv\bmu}{2}\\
    		&\le CR_T^\frac{1}{2}( \slnorm{\sqrt{\rho} \dot{\bmu}}{2} +\slnorm{\nabla^3 \bmd}{2} +\slnorm{\nabla^2 \bmd}{2})(1+R_T^{\frac{\gamma-2\beta}{2}}) \slnorm{\sqrt{2\mu+\lambda} \tdiv\bmu}{2}\\
    		&\quad+C(1+R_T^{\frac{\beta-\gamma}{2}})(1+R_T^{\frac{\gamma-2\beta}{2}})\|\sqrt{2\mu+\lambda}\tdiv\bmu\|_{L^2}^2\\
%    		& \le (\slnorm{\nabla F}{2}+|\bar{F}|)(1+R_T^{\frac{\gamma-2\beta}{2}}) \slnorm{\sqrt{2\mu+\lambda} \tdiv\bmu}{2}\\
%    		&\le (\slnorm{\nabla F}{2}+\slnorm{\rho}{\beta}^{\frac{\beta}{2}}\slnorm{\sqrt{\mu+\lambda} \tdiv\bmu}{2})(1+R_T^{\frac{\gamma-2\beta}{2}}) \slnorm{\sqrt{2\mu+\lambda} \tdiv\bmu}{2}\\
%    		&\le \slnorm{\nabla F}{2}(1+R_T^{\frac{\gamma-2\beta}{2}}) \slnorm{\sqrt{\mu+\lambda} \tdiv\bmu}{2} +\slnorm{\rho}{\beta}^{\frac{\beta}{2}}(1+R_T^{\frac{\gamma-2\beta}{2}})\slnorm{\sqrt{\mu+\lambda} \tdiv\bmu}{2}^2\\
%    		&\le  \frac{1}{8} (\slnorm{\sqrt{\rho }\dot{\bmu}}{2}^2+\slnorm{\nabla^3 \bmd}{2}^2+\slnorm{\nabla^2 \bmd}{2}^2) \\
%    		&\quad+(R_T(1+R_T^{\frac{\gamma-2\beta}{2}})^2+\slnorm{\rho}{\beta}^{\frac{\beta}{2}}(1+R_T^{\frac{\gamma-2\beta}{2}}) )\slnorm{\sqrt{\mu+\lambda} \tdiv\bmu}{2}^2\\
%    		&\le  \frac{1}{8} (\slnorm{\sqrt{\rho }\dot{\bmu}}{2}^2+\slnorm{\nabla^3 \bmd}{2}^2+\slnorm{\nabla^2 \bmd}{2}^2) \\
%    		&\quad+C R_T ((1+R_T^{\frac{\gamma-2\beta}{2}})^2+\slnorm{\rho}{\beta}^{\beta} R^{-2})\slnorm{\sqrt{\mu+\lambda} \tdiv\bmu}{2}^2 \\
    		&\le \frac{1}{4} \slnorm{\sqrt{\rho }\dot{\bmu}}{2}^2+\frac{1}{16}\slnorm{\nabla^3 \bmd}{2}^2+C\slnorm{\nabla^2 \bmd}{2}^2\\
    		&\quad+C R_T (1+R_T^{\beta-\gamma-2} +R_T^{\gamma-2\beta})\slnorm{\sqrt{2\mu+\lambda} \tdiv\bmu}{2}^2.  
%    		&\le  \frac{1}{8} (\slnorm{\sqrt{\rho }\dot{\bmu}}{2}^2+\slnorm{\nabla^3 \bmd}{2}^2+\slnorm{\nabla^2 \bmd}{2}^2) +C R_T (R_T^{\alpha \beta} +R_T^{\gamma-2\beta})\slnorm{\sqrt{\mu+\lambda} \tdiv\bmu}{2}^2. 
    	\end{aligned}
    \end{equation}
Since $\int\tdiv\bmu\, dx=0$,  we can obtain from the definition of $G$,  the H\"{o}lder inequality that
\begin{equation}
	\begin{aligned}
		I_3&=C\left|\int(P-\bar{P})\tdiv\bmu dx\cdot\int\frac{P-\bar{P}}{2\mu+\lambda}dx\right|\\
		&=C\left|\int(G-(2\mu+\lambda)\tdiv\bmu)\tdiv\bmu dx\right|\left|\int\frac{P-\bar{P}}{2\mu+\lambda}dx\right|\\
		&\leq C\|G\|_{L^2}\|\sqrt{2\mu+\lambda}\tdiv\bmu\|_{L^2}+C\|\sqrt{2\mu+\lambda}\tdiv\bmu\|_{L^2}^2\\
		&\leq \frac{1}{4} \slnorm{\sqrt{\rho }\dot{\bmu}}{2}^2+\frac{1}{16} \slnorm{\nabla^3 \bmd}{2}^2+C\slnorm{\nabla^2 \bmd}{2}^2+CR_T(1+R_T^{\frac{\beta-\gamma-2}{2}})\slnorm{\sqrt{2\mu+\lambda} \tdiv\bmu}{2}^2.
	\end{aligned}
\end{equation}
    It follows from the Young and Sobolev inequality,  \eqref{est_d_1} and \eqref{idd} that
    \begin{equation}
    	\begin{aligned}
    		I_4&\leq\frac{1}{16}\|\nabla^3\bmd\|_{L^2}^2+C \|\nabla \bmu\|_2^2\|\nabla\bmd\|_{L^\infty}^2\\
    		&\leq \frac{1}{16}\|\nabla^3\bmd\|_{L^2}^2+C\|\tdiv \bmu\|_2^2\|\nabla\bmd\|_{L^4}\|\nabla^2\bmd\|_{L^4}\\
    		&\leq \frac{1}{16}\|\nabla^3\bmd\|_{L^2}^2+C\|\tdiv \bmu\|_2\|\sqrt{2\mu+\lambda}\tdiv\bmu\|_{L^2}(\slnorm{\nabla^2 \bmd}{2}^{\frac{1}{2}}\slnorm{\nabla^3 \bmd}{2}^{\frac{1}{2}}+\slnorm{\nabla^2 \bmd}{2})\\
    		&\leq  \frac{1}{8}\|\nabla^3\bmd\|_{L^2}^2+C\|\nabla^2\bmd\|_{L^2}^2+C\left\|\frac{G}{\sqrt{2\mu+\lambda}}\right\|_{L^2}^2\|\sqrt{2\mu+\lambda}\tdiv\bmu\|_{L^2}^2\\
    		&\quad+C(R_T^{\gamma-2\beta}+1)\|\sqrt{2\mu+\lambda}\tdiv\bmu\|_{L^2}^2, 
    	\end{aligned}
    \end{equation}
where we have used 
\begin{equation}\label{ellp}
\|\tdiv\bmu\|_{L^2}^2\leq \frac{1}{\mu}\left\|\frac{G}{\sqrt{2\mu+\lambda}}\right\|_{L^2}^2+CR_T^{\gamma-2\beta}+C. 
\end{equation}

    Similarly,  
    \begin{equation}
    	\begin{aligned}
    		I_5&\leq C\|\nabla\bmd\|_{L^\infty}^2\|\nabla\bmd\|_{L^4}^2\|\nabla\bmu\|_{L^2}\\
    		&\leq C(\slnorm{\nabla^2 \bmd}{2}^{\frac{1}{2}}\slnorm{\nabla^3 \bmd}{2}^{\frac{1}{2}}+\slnorm{\nabla^2 \bmd}{2})\|\sqrt{2\mu+\lambda}\tdiv\bmu\|_{L^2}\\
    		&\leq\frac{1}{16}\|\nabla^3\bmd\|_{L^2}^2+C\|\nabla^2\bmd\|_{L^2}^2+C\|\sqrt{2\mu+\lambda}\tdiv\bmu\|_{L^2}^2. 
    	\end{aligned}
    \end{equation}
    and
    \begin{equation}\label{I6-est}
    	\begin{aligned}
    		I_6&\leq C\|\nabla\bmd\|_{L^\infty}^2\|\nabla\bmu\|_{L^2}^2\\
    		&\leq \frac{1}{16}\|\nabla^3\bmd\|_{L^2}^2+C\|\nabla^2\bmd\|_{L^2}^2+C\left\|\frac{G}{\sqrt{2\mu+\lambda}}\right\|_{L^2}^2\|\sqrt{2\mu+\lambda}\tdiv\bmu\|_{L^2}^2\\
    		&\quad+C(R_T^{\gamma-2\beta}+1)\|\sqrt{2\mu+\lambda}\tdiv\bmu\|_{L^2}^2. 
    	\end{aligned}
    \end{equation}
	Plugging estimates of $I_1$-$I_6$ into \eqref{estAB_proof_1} yields that
\begin{equation}\label{estAB_proof_4}
	\begin{aligned}
		&\frac{d}{dt}\left(\frac{1}{2}\left\|\frac{G}{\sqrt{2\mu+\lambda(\rho)}}\right\|_{L^2}^2-I_0\right)  +\frac{1}{4}\lnorm{\sqrt{\rho}\dot{\bmu}}{2}{\Omega}^2\\
		&\leq \frac{7}{16}\|\nabla^3\bmd\|_{L^2}^2+C R_T (1+R_T^{\beta-\gamma-2} +R_T^{\gamma-2\beta})(\slnorm{\sqrt{2\mu+\lambda} \tdiv\bmu}{2}^2+\|\nabla^2\bmd\|_{L^2}^2)\\
		&\quad+CR_T^{\alpha\beta+1}\slnorm{\sqrt{2\mu+\lambda} \tdiv\bmu}{2}^2\left\|\frac{G}{\sqrt{2\mu+\lambda}}\right\|_{L^2}^2.
	\end{aligned}
\end{equation}
% Notice that under the radially symmetric assumptions , we can derive
% \begin{equation}
% 	\begin{aligned}
% 		(\nabla \bmd \odot \nabla \bmd ) : \nabla \bmu & = \partial_i d_k \partial_j d_k \partial_i u_j = \partial_r d_k \partial_r d_k \frac{x_i x_j}{r^2} \partial_i(u(r)\frac{x_j}{r})\\
% 		&=\partial_r d_k \partial_r d_k (\frac{x_i x_j}{r^2} (\partial_r u \frac{x_i x_j}{r^2} + u(\frac{\delta_{ij}}{r}-\frac{x_i x_j}{r^3} )))\\
% 		&=\partial_r d_k \partial_r d_k (\partial_r u + \frac{1}{r}u)=|\nabla \bmd|^2 \tdiv \bmu. 
% 	\end{aligned}
% \end{equation}
On the other hand,  multiplying \eqref{nabula_cel3} by $-\nabla \Delta \bmd $ and integrating the resultant over $\Omega$,  it follows from integration by parts and \eqref{BC} that
	\begin{equation}\label{estAB_proof_2_o}
		\begin{aligned}
			&\frac{1}{2}\frac{d}{dt} \int |\nabla^2\bmd|^2 dx +\int |\nabla \Delta \bmd|^2 dx  \\
			&= \int \nabla(\bmu \cdot \nabla \bmd) \cdot \nabla \Delta \bmd dx-\int \nabla(|\nabla \bmd|^2\bmd)\cdot \nabla \Delta \bmd dx \\
			&=\int (\nabla \bmu \cdot \nabla \bmd) \cdot \nabla \Delta \bmd dx+\int u_i \partial_i \partial_j \bmd \cdot \partial_j \partial_k \partial_k \bmd dx-\int \nabla(|\nabla \bmd|^2\bmd)\cdot \nabla \Delta \bmd dx \\
			&=\int (\nabla \bmu \cdot \nabla \bmd) \cdot \nabla \Delta \bmd dx - \int \partial_ku_i \partial_i \partial_j \bmd \cdot \partial_j \partial_k \bmd dx .
		\end{aligned}
	\end{equation}
	Using elliptic estimates with Neumann boundary conditions,  i. e.  \eqref{elliptic_est2},  we have 
	\begin{equation*}
		\begin{aligned}
			\int|\nabla^3 \bmd|^2 dx\le C \int |\nabla \Delta \bmd|^2 dx , 
		\end{aligned}
	\end{equation*}
	which combining with \eqref{estAB_proof_2_o} yields that
	\begin{equation}\label{estAB_proof_2}
		\begin{aligned}
			&\frac{d}{dt} \int |\nabla^2\bmd|^2 dx +\int |\nabla^3 \bmd|^2 dx\\
			&\le C\int |\nabla \bmu| |\nabla \bmd||\nabla \Delta \bmd| dx +C\int |\nabla \bmu| |\nabla^2 \bmd|^2 dx \\
			&\quad+C\int |\nabla \bmd||\nabla^2\bmd||\nabla \Delta \bmd| dx+C\int |\nabla \bmd|^3|\nabla \Delta \bmd| dx=:\sum_{i=1}^4J_i. 
		\end{aligned}
	\end{equation}
	It follows from the Young and Sobolev inequality,  \eqref{est_d_1} and \eqref{ellp} that
	\begin{equation}
		\begin{aligned}
			J_1&\leq\frac{1}{16}\|\nabla^3\bmd\|_{L^2}^2+C \|\nabla \bmu\|_2^2\|\nabla\bmd\|_{L^\infty}^2\\
			&\leq \frac{1}{16}\|\nabla^3\bmd\|_{L^2}^2+\|\nabla \bmu\|_2^2\|\nabla\bmd\|_{L^4}\|\nabla^2\bmd\|_{L^4}\\
			&\leq \frac{1}{16}\|\nabla^3\bmd\|_{L^2}^2+\|\nabla \bmu\|_2^2(\slnorm{\nabla^2 \bmd}{2}^{\frac{1}{2}}\slnorm{\nabla^3 \bmd}{2}^{\frac{1}{2}}+\slnorm{\nabla^2 \bmd}{2})\\
			&\leq \frac{1}{8}\|\nabla^3\bmd\|_{L^2}^2+C\|\nabla^2\bmd\|_{L^2}^2+C(R_T^{\gamma-2\beta}+1+ \slnorm{\frac{G}{\sqrt{2\mu+\lambda}}}{2}^2)\|\sqrt{2\mu+\lambda}\tdiv\bmu\|_{L^2}^2. 
		\end{aligned}
	\end{equation}
	\begin{equation}
		\begin{aligned}
			J_2&\leq \|\nabla\bmu\|_{L^2}\|\nabla^2\bmd\|_{L^4}^2\leq \|\sqrt{2\mu+\lambda}\tdiv\bmu\|_{L^2}(\|\nabla^2\bmd\|_{L^2}\|\nabla^3\bmd\|_{L^2}+\|\nabla^2\bmd\|_{L^2}^2)\\
			&\leq\frac{1}{16}\|\nabla^3\bmd\|_{L^2}^2+C\|\nabla^2\bmd\|_{L^2}^2\|\sqrt{2\mu+\lambda}\tdiv\bmu\|_{L^2}^2+C\|\nabla^2\bmd\|_{L^2}^4+C\|\sqrt{2\mu+\lambda}\tdiv\bmu\|_{L^2}^2, 
		\end{aligned}
	\end{equation}
	\begin{equation}
		\begin{aligned}
			J_3&\leq C\|\nabla\bmd\|_{L^4}\|\nabla^2\bmd\|_{L^4}\|\nabla^3\bmd\|_{L^2}\\
			&\leq C(\|\nabla^2\bmd\|_{L^2}^\frac{1}{2}\|\nabla^3\bmd\|_{L^2}^\frac{1}{2}+\|\nabla^2\bmd\|_{L^2})\|\nabla^3\bmd\|_{L^2}\\
			&\leq\frac{1}{16}\|\nabla^3\bmd\|_{L^2}^2+C\|\nabla^2\bmd\|_{L^2}^4+C\|\nabla^2\bmd\|_{L^2}^2. 
		\end{aligned}
	\end{equation}
	\begin{equation}
		\begin{aligned}
			J_4\leq\frac{1}{16}\|\nabla^3\bmd\|_{L^2}^2+C\|\nabla\bmd\|_{L^6}^6\leq \frac{1}{16}\|\nabla^3\bmd\|_{L^2}^2+C\|\nabla\bmd\|_{L^2}^2\|\nabla^2\bmd\|_{L^2}^4\leq \frac{1}{16}\|\nabla^3\bmd\|_{L^2}^2+C\|\nabla^2\bmd\|_{L^2}^4.
		\end{aligned}
	\end{equation}
    Plugging the above estimates of $J_1$-$J_4$ into \eqref{estAB_proof_2} yields that
    \begin{equation}\label{estAB_proof_3}
        \begin{aligned}
         &\frac{d}{dt} \int |\nabla^2\bmd|^2 dx +\frac{11}{16}\int |\nabla^3 \bmd|^2 dx\\
		 &\le  C\slnorm{\nabla^2 \bmd}{2}^2 +C\slnorm{\nabla^2 \bmd}{2}^4 +C \slnorm{\frac{G}{\sqrt{2\mu+\lambda}}}{2}^2\|\sqrt{2\mu+\lambda}\tdiv\bmu\|_{L^2}^2\\
		 &\quad+C(R_T^{\gamma-2\beta}+1) \|\sqrt{2\mu+\lambda}\tdiv\bmu\|_{L^2}^2. 
        \end{aligned}
    \end{equation}
    which adding with \eqref{estAB_proof_4} and using the definitions of $A,  B$ gives
%    \begin{equation}
%        \begin{aligned}
%          &\frac{d}{dt}\left(  \frac{1}{2}\left\|\frac{G}{\sqrt{2\mu+\lambda(\rho)}}\right\|_{L^2}^2+\|\nabla^2\bmd\|_{L^2}^2 \right)
%		   +\frac{1}{4}(\lnorm{\sqrt{\rho}\dot{\bmu}}{2}{\Omega}^2+\|\nabla^3\bmd\|_{L^2}^3)\\
%          &\leq \frac{d}{dt} I_0+ C R_T (1+R_T^{\beta-\gamma-2} +R_T^{\gamma-2\beta})(\slnorm{\sqrt{2\mu+\lambda} \tdiv\bmu}{2}^2+\|\nabla^2\bmd\|_{L^2}^2)\\
%		&+CR_T^{\alpha\beta+1}(\slnorm{\frac{G}{\sqrt{2\mu+\lambda}}}{2}^2 + \slnorm{\nabla^2 \bmd}{2}^2)(\slnorm{\sqrt{2\mu+\lambda} \tdiv\bmu}{2}^2+\slnorm{\nabla^2 \bmd}{2}^2). 
%        \end{aligned}
%    \end{equation}
%	bstituting the definitions of  $A, B$ that 
	\begin{equation}\label{estAB_proof_6}
		\begin{aligned}
			&\frac{1}{2} \frac{d}{dt}(e+A^2(t)) + \frac{1}{4} B(t)^2 \\
			&\leq \frac{d}{dt} I_0 +CR_T (1+R_T^{\beta-\gamma-2} +R_T^{\gamma-2\beta})(\slnorm{\sqrt{2\mu+\lambda} \tdiv\bmu}{2}^2+\|\nabla^2\bmd\|_{L^2}^2)\\
			&\quad+CR_T^{\alpha\beta+1}(\slnorm{\sqrt{2\mu+\lambda} \tdiv\bmu}{2}^2+\|\nabla^2\bmd\|_{L^2}^2)(e+A^2(t)). 
		\end{aligned}
	\end{equation}
	Integrating \eqref{estAB_proof_6} in time from $0$ to $t$ and using \eqref{B-E-estimate},  \eqref{d-est} give that
	\begin{equation}\label{estAB_proof_7}
		\begin{aligned}
			&\frac{1}{2} (e+A^2(t)) +\frac{1}{4} \int_0^t B(t)^2\\
			& \leq I_0(t)-I_0(0) +\frac{1}{2}A^2(0) +CR_T (1+R_T^{\beta-\gamma-2} +R_T^{\gamma-2\beta})\\
			&\quad+CR_T^{\alpha\beta+1}\int_{0}^{t}(\slnorm{\sqrt{2\mu+\lambda} \tdiv\bmu}{2}^2+\|\nabla^2\bmd\|_{L^2}^2)(e+A(\tau)^2)dt\\
			&\leq \frac{1}{4} A^2(t) +CR_T (1+R_T^{\beta-\gamma-2} +R_T^{\gamma-2\beta})\\
			&\quad+CR_T^{\alpha\beta+1}\int_{0}^{t}(\slnorm{\sqrt{2\mu+\lambda} \tdiv\bmu}{2}^2+\|\nabla^2\bmd\|_{L^2}^2)(e+A(\tau)^2)dt. 
		\end{aligned}
	\end{equation}
	where we have used \eqref{idd} and \eqref{ellp} to yield the following estimate of $I_0$ 
	\begin{equation}\label{I_0-est}
			|I_0|\leq C\slnorm{\nabla \bmd}{4}^2 \slnorm{\tdiv \bmu}{2} \leq \frac{\mu}{4} \slnorm{\tdiv \bmu}{2}^2 +C\leq \frac{1}{4}\slnorm{\frac{G}{\sqrt{2\mu+\lambda}}}{2}^2 +C R_T^{\gamma-2\beta}+C R_T. 
	\end{equation}
	Applying Gronwall's inequality \eqref{integration-type-gronwall} to \eqref{estAB_proof_7} and using \eqref{B-E-estimate},  \eqref{d-est} imply that
	\begin{equation}\label{estAB_proof_10}
		\begin{aligned}
			 e+A^2(t) \leq CR_T (1+R_T^{\beta-\gamma-2} +R_T^{\gamma-2\beta}) e^{CR_T^{\alpha\beta+1}}
		\end{aligned}
	\end{equation}
	holds for any $0\leq t\leq T$,  which gives
	\begin{equation}\label{estAB_proof_11}
		\begin{aligned}
			\sup_{0\leq t \leq T} \ln (e+A^2(t)) &\leq \ln(CR_T (1+R_T^{\beta-\gamma-2} +R_T^{\gamma-2\beta}))+ CR_T^{\alpha\beta+1}\leq CR_T^{\alpha\beta+1},
		\end{aligned}
	\end{equation}
	where the constant $C$ depends on $\mu, \gamma, \beta, \nabla \bmd_0, \nabla \bmu_0, \rho_0$ and $\alpha$, but is independent to $T$. 

Finally,  plugging \eqref{estAB_proof_10} into \eqref{estAB_proof_7} yields \eqref{est_AB_another}. 
	\end{proof}
\end{lemma}
Now, we can prove the time-independent upper bounds of the density. Following the ideas in \cite{HuangXDLiJ2016,HuangSu2024}, we will reformulate the system \eqref{sCEL-sys} to a transport type equation and employ Zlotnik's inequality in Lemma \ref{Zlotnik} to this new formulation.  To do so, we first define 
\begin{equation}\label{sita_def}
	\theta(\rho)=2\mu \ln\rho+\frac{1}{\beta}(\rho^\beta-1), 
\end{equation}
and
\begin{equation}\label{h_def}
	h(r, t)=\frac{1}{2}{|\partial_r \bmd|}^2+\int_R^r \frac{{|\partial_s \bmd|}^2}{s}ds. 
\end{equation}
By  using $\eqref{sCEL-sys}_1$,  we have 
\begin{equation}\label{S3_linear_eq_1}
	\theta_t+u\partial_r \theta+G+P-\bar{P}=0. 
\end{equation}
Next,  thanks to the definition of $G(r, t)$ and $h(r, t)$ in \eqref{G-def} and \eqref{h_def} respectively,  we can rewrite $\eqref{sCEL-sys}_2$ as
\begin{equation}\label{S3_momentum_eq}
	(\rho u)_t+\partial_r(\rho u^2)+\frac{1}{r}\rho u^2=\partial_rG-\partial_rh. 
\end{equation}
Integrating \eqref{S3_momentum_eq} over $(R, r)$ gives $$\frac{d}{dt} \int_R^r\rho u ds +\rho u^2 + \int_R^r\frac{\rho u^2}{s}ds=G(r, t)-G(R, t)-h(r, t). $$
It should be noticed that $G(R, t)\neq 0$ due to the Dirichlet boundary condition for the velocity $\bmu$, but $h(R, t)=0$ due to the Neumann boundary condition for the director $\bmd$.  

Let
\begin{equation}\label{kexi_yita_def}
	\xi(r,t)=\int_R^r\rho u ds, \quad\eta(r,t)=\rho u^2 + \int_R^r\frac{\rho u^2}{s}ds. 
\end{equation}
Then \eqref{S3_momentum_eq} reduce to  $$\xi_t+ \eta -G+G(R, t)+h=0, $$ which adding to \eqref{S3_linear_eq_1} gives 
\begin{equation}\label{S3_linear_eq_2}
	(\theta +\xi)_t+u \partial_r(\theta +\xi)+\int_R^r\frac{\rho u^2}{s}ds +P-\bar{P}+G(R, t)+h=0. 
\end{equation}
Define
\begin{equation}\label{yg_defin}
	y=\theta(\rho), \quad g(y)=-P(\theta^{-1}(y)),
\end{equation}
 we can further rewrite \eqref{S3_linear_eq_2} into 
\begin{equation}\label{Zlot_form}
	\frac{D}{Dt} y=g(y) -\frac{D}{Dt}b-\frac{d}{dt}Z, 
\end{equation}
with
\begin{equation}\label{b_defin_inZ}
	\begin{aligned}
		b=\xi+\int_0^t\left(G(R, \tau)-\bar{P}(\tau)\right)d\tau, \quad Z=\int_0^t \int_R^r\frac{\rho u^2}{s}dsd\tau  +\int_0^t h d\tau. 
	\end{aligned}
\end{equation}
The following time-independent upper bounds of the density follows from applying Zlotnik's inequality to \eqref{Zlot_form}.
\begin{proposition}\label{time-independent upperbdd_rho}
	Let $\beta,\gamma$ satisfy \eqref{TH1. 1-1}.  Then there is a constant $C>0$, depending only on $\alpha, \mu, \beta, \gamma$, $\Omega$, $\|\rho_0\|_{L^\infty}$, $\|\bmu_0\|_{H^1}$ and $\|\nabla\bmd_0\|_{H^1}$ but not on time $T$,  such that
	\begin{equation}\label{rho_upperbdd}
		\begin{aligned}
			\Vert \rho \Vert_{L^{\infty} L^{\infty}} \le C. 
		\end{aligned}
	\end{equation}
\end{proposition}
\begin{proof}
	It is obvious that $y=\theta(\rho)$ is strictly increasing for $\rho\in(0, +\infty)$ so that $\rho=\theta^{-1}(y)$ exists for $y\in\mathbb{R}$ and $\lim\limits_{y\rightarrow+\infty}g(y)=-\infty. $ It remains to verify that each term in \eqref{Zlot_form} satisfies the conditions in Lemma \ref{Zlotnik},  especially the \textit{Lipschitz} coefficients of $b$ and $Z$. 
	
	By using Lemma \ref{Brezis-Wainger}, \eqref{ellp},  and \eqref{est_AB}, we have that for any $\alpha \in (0, 1)$, 
	\begin{equation}\label{xi_infity}
		\begin{aligned}
			\slnorm{\xi}{\infty} &\le \slnorm{\nabla \xi }{2} \log^\frac{1}{2}(e+\slnorm{\nabla \xi}{3})+\slnorm{\xi}{2}+1\\
			&\le \slnorm{\rho \bmu}{2} \log(e+\slnorm{\rho 
				\bmu}{3}) +\slnorm{\rho \bmu}{\frac{2\gamma}{1+\gamma}}+1\\
			&\le R_T^{\frac{1}{2}} \log(e+R_T\slnorm{\nabla \bmu}{2})+1\\
			&\leq R_T^{\frac{1}{2}} \log(e+R_T(\left\|\frac{G}{\sqrt{2\mu+\lambda}}\right\|_{L^2}+R_T^{\frac{\gamma-2\beta}{2}}+1))+R_T\\
			&\le C(\alpha)R_T^{\frac{1}{2}}R_T^{\frac{1+\alpha \beta}{2}} +R_T\le C(\alpha)R_T^{1+\alpha \beta},
		\end{aligned}
	\end{equation}
which implies that, for any $0 <t_1< t_2$, 
\begin{equation}\label{est-xi}
	|\xi(t_2)-\xi(t_1)|
	\le 2\slnorm{\xi}{\infty} \le C(\alpha)R_T^{1+\alpha \beta} . 
\end{equation}
For the term involving $F(R,\tau):=G(R, \tau)-\bar{P}(\tau)$, we firstly represent $F(R,\tau)$ by some integrals on $\Omega$ as in \cite{HuangSu2024}. Multiplying \eqref{S3_momentum_eq} by $r^2$ gives $$ (r^2 \rho u)_t+\partial_r(r^2 \rho u^2)-{r}\rho u^2=r^2 \partial_r(G-\bar{P})-r^2 \partial_rh, $$
and then integrating over $(0, R)$ leads to $$ \int_0^R (r^2 \rho u)_tdr+\int_0^R \partial_r(r^2 \rho u^2)dr-\int_0^R{r}\rho u^2dr=\int_0^Rr^2 \partial_r(G-\bar{P})dr-\int_0^Rr^2 \partial_rhdr. $$
It follows from \eqref{BC2},  integration by parts,  and $h(R, t)=0$ that
$$\int_0^R\partial_r(r^2 \rho u^2)dr=(r^2 \rho u^2)|_0^R=0, $$ 
$$\int_0^Rr^2 \partial_r(G-\bar{P})dr = r^2(G-\bar{P})|_0^R-\int_0^R 2r (G-\bar{P})dr=R^2F(R, t)-2\int_0^Rr(G-\bar{P})dr, $$
and
$$\int_0^Rr^2 \partial_rhdr = r^2h|_0^R-\int_0^R 2r hdr=-2\int_0^Rrhdr. $$
Thus,  one has 
\begin{equation}\label{FR-def}
	F(R, t)=\frac{1}{R^2}\left[\frac{d}{dt} \int_0^R \rho u r^2dr+2 \int_0^R (G-\bar{P}) rdr -2\int_0^R hrdr-\int_0^R r\rho u^2dr\right]. 
\end{equation}
For any $0 <t_1< t_2$, we have
\begin{equation}\label{integral of FR}
	\begin{aligned}
		\int_{t_1}^{t_2}F(R,\tau)d\tau=&\frac{1}{R^2}\left[\int_0^R \rho u r^2dr|_{\tau=t_2}-\int_0^R \rho u r^2dr|_{\tau=t_1}\right]-\frac{1}{R^2} \int_{t_1}^{t_2}\int_0^R r\rho u^2drd\tau\\
		&+\frac{2}{R^2}\int_{t_1}^{t_2}\int_0^R (G-\bar{P}) rdr-\frac{2}{R^2}\int_{t_1}^{t_2}\int_0^R h rdr.
	\end{aligned}
\end{equation}
It is obvious that
\begin{align}\label{FR1}
	\left|\int_0^R \rho u r^2dr|_{\tau=t_1, t_2}\right|=\frac{1}{2\pi }\left|\int_\Omega\rho urdx\right|\leq CR\int_\Omega(\rho u^2+\rho)dx\leq CR,
\end{align}
and 
\begin{equation}\label{FR2}
	\left|\int_{t_1}^{t_2}\int_0^R r\rho u^2drd\tau\right|\leq \left|\int_{t_1}^{t_2}\int_\Omega\rho u^2dxd\tau\right|\leq C(t_2-t_1). 
\end{equation}
Since 
\begin{equation*}
	\lambda_t+\mathrm{div}(\lambda\bmu)+(\beta-1)\lambda\mathrm{div}\bmu dx=0,
\end{equation*}
which yields that
$$\frac{d}{dt}\int_\Omega\rho^\beta dx=(1-\beta)\int_\Omega\lambda\mathrm{div}\bmu dxd\tau.$$
Then, we have
\begin{equation}\label{FR3}
	\begin{aligned}
	\left|\int_{t_1}^{t_2}\int_0^R (G-\bar{P}) rdr\right|&=\left|\frac{1}{\pi}\int_{t_1}^{t_2}\int_\Omega (G-\bar{P})dx\right|\\
	&=\left|\frac{1}{\pi}\int_{t_1}^{t_2}\int_\Omega(2\mu+\lambda)\mathrm{div}\bmu dxd\tau-\frac{1}{\pi}\int_{t_1}^{t_2}\int_\Omega Pdxd\tau\right|\\
	&=\left|\frac{1}{(1-\beta)\pi}\int\rho^\beta dx|_{t_1}^{t_2}-\frac{1}{\pi}\int_{t_1}^{t_2}\int_\Omega Pdxd\tau\right|\\
	&\leq CR_T^{(\beta-\gamma)_+}+C(t_2-t_1),
\end{aligned}
\end{equation}
where we have used
\begin{equation*}
	\int_\Omega\rho^\beta dx\leq \begin{cases}
		C,& \beta\leq \gamma,\\
		 CR_T^{\beta-\gamma},&\beta\geq\gamma.
	\end{cases}
\end{equation*}
By using interchange the order of integration,  one has
\begin{equation}\label{est_int_h}
	\begin{aligned}
		\left|\int_{t_1}^{t_2} \int_0^R h \cdot rdrd\tau\right|&\le \int_{t_1}^{t_2} \int_0^R {|\partial_r\bmd|}^2 \cdot rdrd\tau +\int_{t_1}^{t_2} \int_0^R |\int_R^r \frac{{|\partial_s \bmd|}^2}{s}ds| rdrd\tau\\
		&\le \frac{1}{2\pi}\int_{t_1}^{t_2} \int_{\Omega} {|\nabla \bmd|}^2dxdt+\int_{t_1}^{t_2} \int_0^R \int_r^R \frac{{|\partial_s \bmd|}^2}{s}ds rdrd\tau\\
		&=\frac{1}{2\pi}\int_{t_1}^{t_2} \int_{\Omega} {|\nabla \bmd|}^2dxdt+\int_{t_1}^{t_2} \frac{1}{2}\int_0^R r^2\frac{{|\partial_s \bmd|}^2}{r}drd\tau\\
		&=\frac{3}{4\pi}\int_{t_1}^{t_2} \int_{\Omega} {|\nabla \bmd|}^2dxdt\le C(t_2-t_1). 
	\end{aligned}
\end{equation}
Thus, plugging \eqref{FR1}-\eqref{est_int_h} into \eqref{integral of FR} leads to
\begin{equation}\label{est-int-G}
	\left|\int_{t_1}^{t_2}G(R,\tau)-\bar{P}(\tau)d\tau\right|\leq C(t_2-t_1)+CR_T^{(\beta-\gamma)_+}+C.
\end{equation}
	Note that
	\begin{equation*}
		\int_\Omega\frac{u^2}{r^2}dr\leq \int_\Omega(\partial_ru)^2+\frac{u^2}{r^2}dr=\|\mathrm{div}\bmu\|_{L^2}^2\leq \|\nabla\bmu\|_{L^2}^2,
	\end{equation*}
we have
\begin{equation}\label{est-int-rhou2}
	\left|\int_{t_1}^{t_2} \int_R^r\frac{\rho u^2}{s}dsd\tau\right|\leq \|\rho\|_{L^\infty L^\infty}\int_{t_1}^{t_2}\int_\Omega\frac{u^2}{r^2}drd\tau\leq CR_T. 
\end{equation}
It obvious that
\begin{equation}\label{est-int-h}
	\begin{aligned}
		\left|\int_{t_1}^{t_2} h d\tau\right|&=\int_{t_1}^{t_2} \left(\frac{1}{2}|\partial_r\bmd|^2+\int_R^r\frac{|\partial_sd|^2}{s}ds \right)d\tau\\
		&\le \int_0^T \int\frac{1}{r^2}|\partial_r\bmd|^2dx+\int | \partial_{rr}\bmd|^2 dxdt +\int_0^T \int_0^R \frac{|\partial_r\bmd|^2}{r}drdt\le C. 
	\end{aligned}
\end{equation}
where we have used \eqref{B-E-estimate} and 
\begin{equation*}
	\begin{aligned}
		|\partial_r \bmd|^2&=2\int_0^r\partial_r\bmd \cdot \partial_{rr}\bmd dr \le 2\int_0^R |\partial_r\bmd|| \partial_{rr}\bmd| dr\\
		&\le \int_0^R\frac{1}{r} |\partial_r\bmd|^2dr +\int_0^R r| \partial_{rr}\bmd|^2 dr\\
		&\le \int\frac{1}{r^2}|\partial_r\bmd|^2dx+\int | \partial_{rr}\bmd|^2 dx. 
	\end{aligned}
\end{equation*}
Therefore,  \eqref{est-xi}, \eqref{est-int-G}, \eqref{est-int-rhou2} and \eqref{est-int-h} leads to
\begin{equation*}
	\begin{aligned}
		|b(t_2)-b(t_1)|+|Z(t_2)-Z(t_1)| \le C+C(\alpha)R_T^{1+\alpha \beta}+CR_T^{(\beta-\gamma)_+}+C(t_2-t_1). 
	\end{aligned}
\end{equation*}
It follows from using \textit{Zlotnik's} inequality \eqref{lemma2-8} that
\begin{equation*}
	\begin{aligned}
		\theta \le C(\alpha)R_T^{1+\alpha \beta}+CR_T^{(\beta-\gamma)_+}, 
	\end{aligned}
\end{equation*}
%		that is same to 
%		\begin{equation*}
	%			\begin{aligned}
		%				\rho^\beta \le C(\alpha)R_T^{1+\alpha \beta} +R_T^{\gamma-2\beta+1}, 
		%			\end{aligned}
	%		\end{equation*}
which implies that
\begin{equation*}
	\begin{aligned}
		R_T^\beta \le C(\alpha)R_T^{1+\alpha \beta}+R_T^{(\beta-\gamma)_+}.
	\end{aligned}
\end{equation*}
Hence, by choosing $0<\alpha<\frac{\beta-1}{\beta}<1$,  we get $	R_T\le C$. This completes the proof of proposition.
\end{proof}

\section{Higher order A prior estimates estimates}
%Indeed,  the bound of the density is enough to extend the local strong solution to a global one.  The proof is analogous to the blowup criterion established in \cite{HuangSu2024}. 
%
%To make the content self-contained,  we give a detailed proof as follows. 
%With the time-independent upper bound of the density at hand, we can also derive the following higher order time-independent estimates.

In this section, we derive higher order estimates based on the upper bound of the density.
\begin{lemma}
	There is a positive constant $C$ depending only on $\mu, \beta, \gamma$, $\Omega$, $\|\rho_{0}\|_{L^\infty}$  $\|\bmu_0\|_{H^1}$, and $\|\nabla\bmd_0\|_{H^1}$ such that
	\begin{equation}\label{est_rhodotu}
\begin{aligned}
		\sup_{0\le t \le T} &(\sigma \slnorm{\sqrt{\rho}\dot{\bmu}}{2}^2 +\sigma \slnorm{\nabla \bmd_t}{2}^2+\sigma\|\nabla^3\bmd\|_{L^2}^2) \\
		&\quad+ \int_0^T \sigma \slnorm{\nabla \dot{\bmu}}{2}^2+\sigma \slnorm{\nabla^2\bmd_t }{2}^2 dt \le C,
\end{aligned}
	\end{equation}
	where $\sigma : =\min\{ 1, t \}$. Moreover, for any $p\in[1,+\infty)$, there exists a positive constant $C(p)$ depending only on $p, \mu, \beta, \gamma, \Omega$, $\|\rho_{0}\|_{L^\infty}$, $\|\bmu_0\|_{H^1}$, and $\|\nabla\bmd_0\|_{H^1}$ such that
	\begin{equation}\label{est_nbdup}
		\sup_{0 \le t \le T}\sigma^\frac{1}{2}\|\nabla\bmu\|_{L^p}+\sigma^\frac{1}{2}\|\nabla^2\bmd\|_{L^p}\leq C(p).
	\end{equation}
	\begin{proof}
		First, it follows from \eqref{est_AB_another}, \eqref{B-E-estimate}, \eqref{idd}, \eqref{d-est2}, \eqref{lemma2-2} and \eqref{rho_upperbdd} that
		\begin{equation}\label{low-ord-est}
			\|\bmu\|_{L^\infty}^2+\|\nabla\bmd\|_{L^\infty}^2+\|\bmu\|_{H^1}^2+\|\nabla\bmd\|_{H^1}^2+\int_{0}^{T}\|\bmu\|_{H^1}^2+\|\sqrt{\rho}\dot{\bmu}\|_{L^2}^2+\|\nabla^2\bmd\|_{H^1}^2dt\leq C,
		\end{equation}
		        where we have used the fact
		\begin{equation*}
			\|\nabla\bmd\|_{L^\infty}^2\leq C\int_\Omega\left(|\partial_{r r}\bmd|^2+\frac{|\partial_{r}\bmd|^2}{r^2}\right)dx\leq \int_\Omega|\Delta\bmd|^2dx\leq C
		\end{equation*}
		due to \eqref{dinfty}, \eqref{syscond_d}, \eqref{est_AB_another} and \eqref{rho_upperbdd}.
		
		Next, applying $\dot{u}^j\left(\frac{\partial}{\partial t}+\text{div}(\bmu \cdot)\right)$ to the $j$-th component of $\eqref{CEL-sys}_2$,  summing with respect to $j$, and integrating the resultant over $\Omega$ yield that
		\begin{equation}\label{rhodtu_p_1}
			\begin{aligned}
				\frac{1}{2}\frac{d}{dt}\left(\int \rho|\dot{\bmu}|^2 dx\right) &= - \int \dot{u}_j\left(\partial_j P_t + \text{div}(\bmu \partial_j P)\right)dx+\mu \int \dot{u}_j\left(\Delta u_{tj} + \text{div}(\bmu \Delta u_j)\right)dx \\
				& \quad + \int \dot{u}_j\left[\partial_t \partial_j\left((\mu + \lambda)\text{div} \bmu\right) + \text{div}\left(\bmu \partial_j\left((\mu + \lambda)\text{div} \bmu\right)\right)\right]dx \\
				& \quad - \int \dot{u}_j\left[\partial_t\partial_kD_{jk}+\tdiv(\bmu \partial_kD_{jk}) \right]dx =:\sum_{i=1}^4 K_i. 
			\end{aligned}
		\end{equation}
		with $D_{jk}=\left(\partial_j\bmd\cdot \partial_k\bmd-\frac{1}{2}|\nabla\bmd|^2\delta_{jk}\right)$.
		
		Note that $\dot{\bmu}\big|_{\partial\Omega}=0$ due to \eqref{BC}. Then, we can get from integration by parts,  $\eqref{CEL-sys}_1$ and \eqref{rho_upperbdd} that
		 \begin{equation*}\label{rhodtu_p_2_1}
			\begin{aligned}
				K_1 &= -\int \dot{u}_j [\partial_j P_t + \tdiv \partial_j( P \bmu)-\tdiv ( P \partial_j\bmu)] dx \\
				&= \int \tdiv\dot{\bmu}(P_t+\tdiv(P\bmu)) dx-\int\nabla\dot{u}_j\cdot\partial_j\bmu Pdx \\
				& \leq C\slnorm{\nabla \dot{\bmu}}{2}\slnorm{\nabla \bmu}{2} \leq \frac{\mu}{8} \slnorm{\nabla \dot{\bmu}}{2}^2 + C\slnorm{\nabla \bmu}{2}^2, 
			\end{aligned}
		\end{equation*}
        \begin{equation*}\label{rhodtu_p_2_2}
    		\begin{aligned}
    			K_2 &= \mu \int \dot{u}_j [\Delta \dot{u}_j-2\partial_iu_k\partial_{ki}^2u_j-\Delta u_k\partial_k u_j + \tdiv\bmu \Delta u_j] dx \\
    			&=-\mu \int (|\nabla \dot{\bmu}|^2-2\partial_i \dot{u}_j\partial_iu_k\partial_k u_j-\dot{u}_j\Delta u_k\partial_k u_j+\partial_i\dot{u}_j\tdiv \bmu \partial_i u_j+\dot{u}_j\partial_i(\tdiv\bmu)\partial_i u_j)dx\\
    			&= -\mu \int (|\nabla \dot{\bmu}|^2 + \partial_i \dot{u}_j \partial_k u_k \partial_i u_j - \partial_i \dot{u}_j \partial_i u_k \partial_k u_j - \partial_i u_j \partial_i u_k \partial_k \dot{u}_j) dx \\
    			&\leq -\frac{7}{8}\mu \slnorm{\nabla \dot{\bmu}}{2}^2 + C \slnorm{\nabla \bmu}{4}^4
    		\end{aligned}
        \end{equation*}
        and
                \begin{equation*}\label{rhodtu_p_2}
        	\begin{aligned}
        		K_3 &= \int \dot{u}_j [\partial_{jt} ((\mu + \lambda)\tdiv \bmu) + \tdiv\partial_j (\bmu  (\mu + \lambda)\tdiv \bmu)-\tdiv (\partial_j\bmu  ((\mu + \lambda)\tdiv \bmu))] dx \\
        		&= -\int \tdiv \dot{\bmu} [((\mu + \lambda)\tdiv \bmu)_t + \tdiv (\bmu (\mu + \lambda)\tdiv \bmu)] dx-\int \dot{u}_j \tdiv (\partial_j \bmu (\mu + \lambda)\tdiv \bmu) dx \\
%        		&= -\int \tdiv \dot{\bmu} [(\mu + \lambda)\tdiv \bmu_t + (\mu + \lambda)(\bmu \cdot \nabla)\tdiv \bmu]dx \\
%        		&\quad-\int \tdiv \dot{\bmu} [\lambda_t \tdiv \bmu + (\bmu \cdot \nabla \lambda)\tdiv \bmu + (\mu + \lambda)(\tdiv \bmu)^2 ]dx \\
%        		&\quad +\int \nabla \dot{u}_j \cdot \partial_j \bmu (\mu + \lambda)\text{div} \bmu dx \\
        		&= -\int (\mu + \lambda)\tdiv \dot{\bmu} (\tdiv \dot{\bmu} - \partial_i u_j \partial_j u_i) dx-\int (\mu + \lambda - \beta \lambda)\tdiv \dot{\bmu} (\tdiv \bmu)^2 dx\\
        		& \quad + \int (\mu + \lambda)\nabla \dot{u}_j \cdot \partial_j \bmu\tdiv \bmu dx \\
        		& \leq -\left\|\sqrt{\mu + \lambda}\tdiv \dot{\bmu}\right\|_{L^2}^2 + \frac{\mu}{8} \slnorm{\nabla \dot{\bmu}}{2}^2 + C\slnorm{\nabla \bmu}{4}^4.
        	\end{aligned}
        \end{equation*}
           
%		 It follows $\eqref{CEL-sys}_3$, integration by parts and $\dot{\bmu}\big|_{\partial\Omega}=0$ that
%        \begin{equation}\label{rhodtu_p_3}
%            \begin{aligned}
%                K_4 &= -\int \tdiv \dot{\bmu} \partial_k \bmd \cdot \partial_k \bmd_t dx + \frac{1}{2} \int \dot{u}_j \bmu\cdot \nabla (\partial_j |\nabla \bmd|^2) dx+ \frac{1}{2} \int \dot{u}_j  \tdiv\bmu \partial_j |\nabla \bmd|^2dx  \\
%				&= -\int \tdiv \dot{\bmu} \partial_k \bmd \cdot \left(\Delta \partial_k \bmd - \partial_k \bmu \cdot \nabla \bmd + \partial_k (|\nabla \bmd|^2 \bmd)\right) dx \\
%				&\quad  - \frac{1}{2} \int \dot{u}_j\partial_j u_i\partial_i |\nabla \bmd|^2 dx+ \frac{1}{2} \int \dot{u}_j \tdiv\bmu \partial_j |\nabla \bmd|^2 dx\\
%				&= -\int \partial_j \dot{u}_j \partial_k \bmd \cdot \left(\Delta \partial_k \bmd - \partial_k \bmu \cdot \nabla \bmd +|\nabla \bmd|^2 \partial_k\bmd\right) dx \\
%				&\quad  +\frac{1}{2} \int \partial_i\dot{u}_j\partial_j u_i |\nabla \bmd|^2 dx- \frac{1}{2} \int \partial_j\dot{u}_j \tdiv\bmu  |\nabla \bmd|^2 dx\\
%				&\leq \frac{\mu}{8} \slnorm{\nabla \dot{\bmu}}{2}^2 +C\|\nabla \bmd\|_{L^\infty}^2\|\nabla \Delta \bmd\|_{L^2}^2+ C \slnorm{\nabla \bmu}{4}^4 + C \slnorm{\nabla \bmd}{8}^8. 
%			\end{aligned}
%		\end{equation}
Similarly, using integration by parts and \eqref{low-ord-est} leads to
		\begin{equation*}\label{rhodtu_p_4}
			\begin{aligned}
				K_4 &=\int \partial_k \dot{u}_j\partial_tD_{jk}+\nabla \dot{u}_j\cdot\bmu \partial_k D_{jk}dx\\
				&\leq C\int|\nabla\dot{\bmu}||\nabla\bmd|(|\nabla\bmd_t|+|\nabla^2\bmd||\bmu|)dx\\
				&\leq\|\nabla\dot{\bmu}\|_{L^2}\|\nabla\bmd\|_{L^\infty}(\|\nabla\bmd_t\|_{L^2}+\|\nabla^2\bmd\|_{L^2}\|\bmu\|_{L^\infty})dx\\
				&\leq   \frac{\mu}{8} \slnorm{\nabla \dot{\bmu}}{2}^2+C\left(\|\nabla\bmd_t\|_{L^2}^2+\|\nabla^2\bmd\|_{L^2}^2\right).
			\end{aligned}
		\end{equation*}
		Plugging estimates of $K_1-K_4$ into \eqref{rhodtu_p_1} leads to 
		\begin{equation}\label{rhodtu_p_5}
			\begin{aligned}
				&\frac{d}{dt} \slnorm{\sqrt{\rho}\dot{\bmu}}{2}^2 +\mu \|\nabla\dot{\bmu}\|_{L^2}^2 +\left\|\sqrt{\mu + \lambda}\tdiv \dot{\bmu}\right\|_{L^2}^2 \\
				&\le C\left(\slnorm{\nabla \bmu}{4}^4+\|\nabla\bmd_t\|_{L^2}^2+\|\nabla^2\bmd\|_{L^2}^2+\slnorm{\nabla \bmu}{2}^2\right).
			\end{aligned}
		\end{equation}
		Using elliptic estimates of \eqref{ellipest_1}, \eqref{G-def} and \eqref{rho_upperbdd}, one has
		\begin{equation}\label{rhodtu_p_8}
			\begin{aligned}
				\slnorm{\nabla \bmu}{4}^4 &\le C\slnorm{\text{div} \bmu}{4}^4 \le C\left(\slnorm{G}{4}^4+\slnorm{P-\bar{P}}{4}^4\right)\\
				&\leq C(\|G\|_{L^2}^2\|G\|_{H^1}^2+\slnorm{P-\bar{P}}{2}^2)\\
				&\leq C\left(\left\|\frac{G}{\sqrt{2\mu+\lambda}}\right\|_{L^2}^2\|G\|_{H^1}^2+\slnorm{G}{2}^2 +\slnorm{\nabla \bmu}{2}^2\right)\\
				&\leq C\left(\slnorm{\sqrt{\rho}\dot{\bmu}}{2}^2 +\|\nabla^2 \bmd\|_{H^1}^2+\slnorm{\nabla \bmu}{2}^2\right). 
			\end{aligned}
		\end{equation}
		where we have used \eqref{est_AB_another} and \eqref{est_G_H1} in the last inequality. 
		
		In view of \eqref{nabula_cel3} and \eqref{low-ord-est}, one has 
		\begin{equation}\label{nbdt-est}
			\begin{aligned}
				\|\nabla\bmd_t\|_{L^2}^2&\leq C\left(\|\nabla^3\bmd\|_{L^2}^2+\|\nabla\bmu\|_{L^2}^2\|\nabla\bmd\|_{L^\infty}^2+(\|\bmu\|_{L^\infty}^2+\|\nabla\bmd\|_{L^\infty}^2)\|\nabla^2\bmd\|_{L^2}^2+\|\nabla\bmd\|_{L^6}^6\right)\\
				&\leq C\left(\|\nabla^3\bmd\|_{L^2}^2+\|\nabla\bmu\|_{L^2}^2+\|\nabla\bmd\|_{L^2}^2\|\nabla^2\bmd\|_{L^2}^4\right)\\
				&\leq C\left(\|\nabla^2\bmd\|_{H^1}^2+\|\nabla\bmu\|_{L^2}^2\right).
			\end{aligned}
		\end{equation}
		Therefore, combing \eqref{rhodtu_p_5}, \eqref{rhodtu_p_8} and \eqref{nbdt-est} implies that
		\begin{equation}
			\begin{aligned}
			\frac{d}{dt} \slnorm{\sqrt{\rho}\dot{\bmu}}{2}^2 +\mu \|\nabla\dot{\bmu}\|_{L^2}^2\le C\left(\slnorm{\sqrt{\rho}\dot{\bmu}}{2}^2 +\|\nabla^2 \bmd\|_{H^1}^2+\slnorm{\nabla \bmu}{2}^2\right). 
			\end{aligned}
		\end{equation}
		Multiplying the above inequality by $\sigma(t)$, integrating over $[0, T]$ , and using \eqref{low-ord-est} yield that
			\begin{equation}\label{est_rhodotu1}
			\sup_{0\le t \le T} \sigma \slnorm{\sqrt{\rho}\dot{\bmu}}{2}^2+\int_0^T \sigma \slnorm{\nabla \dot{\bmu}}{2}^2dt \le C(M).
		\end{equation}
		Note that $\nabla \bmd_j|_{\partial\Omega}=\partial_rd_j\frac{x}{r}|_{\partial\Omega}=0$, which follows from \eqref{BC2}. Then, differentiating \eqref{nabula_cel3} respect to $t$, multiplying the resultant by $\nabla \bmd_t$, and using integration by parts, we obtain that 
%		\begin{equation}\label{nabula_t_cel3}
%			\begin{aligned}
%				\nabla \bmd_{tt} -\Delta \nabla \bmd_t =-\nabla (\bmu \cdot \nabla \bmd)_t +\nabla (|\nabla \bmd|^2\bmd)_t. 
%			\end{aligned}
%		\end{equation}
%		, we arrive that
\begin{equation}\label{nbdt-ests}
	\begin{aligned}
	\frac{1}{2}\frac{d}{dt}\slnorm{\nabla \bmd_t}{2}^2+\slnorm{\nabla^2 \bmd_t}{2}^2 &=-\int \nabla (\bmu \cdot \nabla \bmd)_t \cdot\nabla\bmd_t dx+\int \nabla (|\nabla \bmd|^2\bmd)_t\cdot \nabla\bmd_t dx\\
		&=-\int\nabla\dot{\bmu}\cdot\nabla\bmd\cdot\nabla\bmd_t dx-\int(\dot{\bmu}\cdot\nabla)\nabla\bmd\cdot\nabla\bmd_t dx\\
		&\quad+\int\nabla((\bmu\cdot\nabla\bmu)\cdot\nabla\bmd)\cdot\nabla\bmd_t dx-\int\nabla(\bmu\cdot\nabla\bmd_t)\cdot\nabla\bmd_t dx\\
		&\quad+\int \nabla (|\nabla \bmd|^2\bmd)_t\cdot \nabla\bmd_t dx=:\sum_{i=1}^5 N_i
	\end{aligned}
\end{equation}
		It follows from \eqref{lemma2-3} that 
		\begin{equation*}
				\begin{aligned}
					N_1 &\leq C\|\nabla \dot{\bmu}\|_{L^2}\|\nabla \bmd\|_{L^4}\|\nabla \bmd_t\|_{L^4}\\
					&\leq C\|\nabla \dot{\bmu}\|_{L^2}\|\nabla \bmd_t\|_{L^2}^\frac{1}{2}\|\nabla^2 \bmd_t\|_{L^2}^\frac{1}{2}\\
					&\leq\frac{1}{8} \|\nabla^2 \bmd_t\|_{L^2}^2+C\|\nabla\bmd_t\|_{L^2}^2+C\|\nabla \dot{\bmu}\|_{L^2}^2.
				\end{aligned}
		\end{equation*}
		By using \eqref{lemma2-1}, one has
		\begin{equation*}
			\begin{aligned}
				N_2&\leq C\|\dot{\bmu}\|_{L^4}\|\nabla \bmd_t\|_{L^4}\|\nabla^2\bmd\|_{L^2}\\
				& \leq C\|\nabla \dot{\bmu}\|_{L^\frac{4}{3}}\|\nabla \bmd_t\|_{L^2}^\frac{1}{2}\|\nabla^2 \bmd_t\|_{L^2}^\frac{1}{2}\\
				& \leq \frac{1}{8}\|\nabla^2 \bmd_t\|_{L^2}^2 + C\|\nabla \bmd_t\|_{L^2}^2+C\|\nabla \dot{\bmu}\|_{L^2}^2.
			\end{aligned}
		\end{equation*}
		Using integration by parts and \eqref{low-ord-est} lead to
		\begin{equation*}
				N_3 \leq \|\bmu\|_{L^\infty}\|\nabla\bmu\|_{L^2}\|\nabla\bmd\|_{L^\infty}\|\nabla^2\bmd_t\|_{L^2}\leq \frac{1}{8} \|\nabla^2 \bmd_t\|_{L^2}^2+C\|\nabla\bmu\|_{L^2}^2.
		\end{equation*}
		\begin{equation*}
				N_4 \leq \|\bmu\|_{L^\infty}\|\nabla\bmd_t\|_{L^2}\|\nabla^2\bmd_t\|_{L^2}\leq \frac{1}{8} \|\nabla^2 \bmd_t\|_{L^2}^2+C\|\nabla\bmd_t\|_{L^2}^2
		\end{equation*}
		\begin{align*}
				N_5&\leq C(\|\nabla\bmd\|_{L^\infty}\|\nabla \bmd_t\|_{L^2}+\|\nabla\bmd\|_{L^\infty}^2\|\bmd_t\|_{L^2})\|\nabla^2 \bmd_t\|_{L^2}\\
				&\leq \frac{1}{8} \|\nabla^2 \bmd_t\|_{L^2}^2+C\|\nabla \bmd_t\|_{L^2}^2+C\|\bmd_t\|_{L^2}\\
				&\leq \frac{1}{8} \|\nabla^2 \bmd_t\|_{L^2}^2+C\|\nabla \bmd_t\|_{L^2}^2+C(\|\Delta\bmd+|\nabla\bmd|^2\bmd\|_{L^2}^2+\|\bmu\|_{H^1}^2),
		\end{align*}
		where we have used the following inequality
		and
		\begin{equation}\label{L2est_dt}
			\begin{aligned}
			\|\bmd_t\|_{L^2}^2&\leq C(\|\Delta\bmd+|\nabla\bmd|^2\bmd\|_{L^2}^2+\|\bmu\|_{L^4}^2\|\nabla\bmd\|_{L^4}^2)\\
			&\leq C(\|\Delta\bmd+|\nabla\bmd|^2\bmd\|_{L^2}^2+\|\bmu\|_{H^1}^2).
		\end{aligned}
		\end{equation}
Plugging estimates of $N_1,\cdots,N_5$ into \eqref{nbdt-ests} yields that
\begin{equation*}
	\frac{d}{dt}\slnorm{\nabla \bmd_t}{2}^2+\slnorm{\nabla^2 \bmd_t}{2}^2\leq C(M)(\|\nabla \dot{\bmu}\|_{L^2}^2+\|\nabla\bmd_t\|_{L^2}^2+\|\bmu\|_{H^1}^2+\|\Delta\bmd+|\nabla\bmd|^2\bmd\|_{L^2}^2).
\end{equation*}
  Multiplying the above inequality by $\sigma(t)$, integrating over $[0, T]$ , and using \eqref{nbdt-est}, \eqref{low-ord-est} and \eqref{est_rhodotu1}, we get that
 			\begin{equation}\label{est_nbdt}
 	\sup_{0\le t \le T} \sigma \slnorm{\nabla \bmd_t}{2}^2+\int_0^T \sigma \slnorm{\nabla^2 \bmd_t}{2}^2dt \le C.
 \end{equation}
Moreover, using standard elliptic estimates \eqref{elliptic_est2}, \eqref{low-ord-est}, \eqref{L2est_dt} leads to 
\begin{equation}
		\|\nabla^3\bmd\|_{L^2}^2\leq C(\|\bmd_t\|_{H^1}^2+\|\bmu\cdot\nabla\bmd\|_{H^1}^2+\||\nabla\bmd|^2\bmd\|_{H^1}^2)\leq C\|\nabla\bmd_t\|_{L^2}^2+C,
\end{equation}
which, multiplying by $\sigma(t)$, taking supremum with respect to $t$ and using \eqref{est_nbdt}, implies \begin{equation}\label{est_3nbd}
	\sup_{0\le t \le T} \sigma \slnorm{\nabla^3 \bmd}{2}^2\le C.
\end{equation}
Finally, the standard $ L^p$-estimates for the elliptic system \eqref{ellipest_1} together with \eqref{rho_upperbdd} and \eqref{eq_dotu} yields that
	\begin{align*}
		\|\nabla\bmu\|_{L^p}&\leq C\|\tdiv\bmu\|_{L^p}\leq C(\|G\|_{L^p}+\|P-\bar{P}\|_{L^p})\leq C\|G\|_{H^1}+C(p)\\
		&\leq C(p)\|\sqrt{\rho}\dot{\bmu}\|_{L^2}+C(p)\|\nabla\bmd\|_{L^\infty}\|\nabla^2\bmd\|_{L^2}+C(p)\\
		&\leq C(p)\|\sqrt{\rho}\dot{\bmu}\|_{L^2}+C(p),
	\end{align*}
and the Sobolev inequality leads to
\begin{equation*}
	\|\nabla^2\bmd\|_{L^p}\leq C\|\nabla^2\bmd\|_{H^1}.
\end{equation*}
These combined with \eqref{est_rhodotu1} and \eqref{est_3nbd} give \eqref{est_nbdup}.
Therefore, the proof of lemma is completed.
	\end{proof}
\end{lemma}

The following lemma provides some necessary estimates to control $\|\nabla\bmu\|_{L^\infty}$ and $\|\nabla\rho\|_{L^q}$. 
\begin{lemma}
	Assume that \eqref{TH1. 1-1} holds. Then for any $q>2$, there exists a constant $C$ depending on $T, q, \mu, \beta, \gamma, \|\rho_0\|_{L^\infty}$,  $\|\bmu_0\|_{H^1}$, and $\|\nabla\bmd_0\|_{H^1}$ such that
	\begin{equation}\label{rhodotuLq-est}
		\int_{0}^{T}(\|G\|_{L^\infty}+\|\nabla G\|_{L^q}+\|\rho\dot{\bmu}\|_{L^q})^{1+\frac{1}{q}}dt+\int_{0}^{T}t(\|\nabla G\|_{L^q}^2+\|\dot{\bmu}\|_{H^1}^2)\leq C.
	\end{equation}
\end{lemma}
\begin{proof}
	It follows from \eqref{Poincare-in} that
	\begin{equation}\label{rhodotuLq}
		\begin{aligned}
		\|\rho\dot{\bmu}\|_{L^q}&\leq C\|\rho\dot{\bmu}\|_{L^2}^{\frac{2(q-1)}{q^2-2}}\|\dot{\bmu}\|_{L^{q^2}}^\frac{q(q-2)}{q^2-2}\\
		&\leq  C\|\rho\dot{\bmu}\|_{L^2}^{\frac{2(q-1)}{q^2-2}}\|\dot{\bmu}\|_{H^1}^\frac{q(q-2)}{q^2-2}\\
		&\leq C\|\rho\dot{\bmu}\|_{L^2}+C\|\rho\dot{\bmu}\|_{L^2}^{\frac{2(q-1)}{q^2-2}}\|\nabla\dot{\bmu}\|_{L^2}^\frac{q(q-2)}{q^2-2},
	\end{aligned}
	\end{equation}
	which together with \eqref{rho_upperbdd}, \eqref{low-ord-est} and \eqref{est_rhodotu} yields that
	\begin{equation}\label{rhodotu}
	\int_{0}^{T}\left(\|\rho\dot{\bmu}\|_{L^q}^{1+\frac{1}{q}}+t\|\dot{\bmu}\|_{H^1}^2\right)dt\leq C\int_{0}^{T}\left(\|\sqrt{\rho}\dot{\bmu}\|_{L^2}^2+t\|\nabla\dot{\bmu}\|_{L^2}^2+t^{-1+\frac{2}{q^3-q^2-2q+2}} +1 \right)dt\leq C(T).
	\end{equation}
		Next, we deduce from \eqref{eq_dotu}, \eqref{lemma2-3} and \eqref{low-ord-est} that
	\begin{equation}\label{nbdG}
		\begin{aligned}
			\|\nabla G\|_{L^q}&\leq C\|\rho\dot{\bmu}\|_{L^q}+C\|\nabla\bmd\|_{L^\infty}\|\nabla^2\bmd\|_{L^q}\\
			&\leq C\|\rho\dot{\bmu}\|_{L^q}+C(\slnorm{\nabla^2 \bmd}{2}^{\frac{2}{q}}\slnorm{\nabla^3 \bmd}{2}^{1-\frac{2}{q}} + \slnorm{\nabla^2 \bmd}{2})\\
			&\leq C\|\rho\dot{\bmu}\|_{L^q}+C\slnorm{\nabla^3 \bmd}{2}^{1-\frac{2}{q}} + C
		\end{aligned}
	\end{equation}
	Thus, one gets from \eqref{rhodotu} and \eqref{low-ord-est} that
	\begin{equation*}
		\begin{aligned}
			\int_{0}^{T}\|\nabla G\|_{L^q}^{1+\frac{1}{q}}dt\leq C\int_{0}^{T}\left(\|\rho\dot{\bmu}\|_{L^q}^{1+\frac{1}{q}}+\slnorm{\nabla^3 \bmd}{2}^{\frac{(q-2)(q+1)}{q^2}}+1\right)dt\leq C(T),
		\end{aligned}
	\end{equation*}
	and 
	\begin{align*}
			\int_{0}^{T}t\|\nabla G\|_{L^q}^2dt&\leq C\int_{0}^{T}t\|\rho\dot{\bmu}\|_{L^q}^2dt+C\int_{0}^{T}t\slnorm{\nabla^3 \bmd}{2}^{\frac{2(q-2)}{q}}dt+C(T)\\
			&\leq C\int_{0}^{T}t\|\sqrt{\rho}\dot{\bmu}\|_{L^2}^2+\left(t^\frac{1}{2}\|\sqrt{\rho}\dot{\bmu}\|_{L^2}\right)^{\frac{4(q-1)}{q^2-2}}\left(t^\frac{1}{2}\|\nabla\dot{\bmu}\|_{L^2}\right)^\frac{2q(q-2)}{q^2-2}dt+C(T)\\
			&\leq C(T).
	\end{align*}
	Finally, in view of \eqref{G-def}, \eqref{lemma2-4}, \eqref{est_AB_another} and \eqref{rho_upperbdd}, one has
	\begin{equation}\label{hdproof_4}
		\begin{aligned}
			\slnorm{\tdiv \bmu}{\infty} &\le \left\|\frac{G+P-\bar{P}}{2\mu+\lambda}\right\|_{L^\infty} \le C(1+\slnorm{G}{\infty})\\
			& \le C(\slnorm{G}{2}^{\frac{q-2}{2q-2} } \slnorm{\nabla G}{q}^{\frac{q}{2q-2}} + \slnorm{G}{2}+1)\\
			&\le C(\slnorm{\nabla G}{q}^{\frac{q}{2q-2}}+1)\\
			& \le C(\slnorm{\rho \dot{\bmu}}{q}^{\frac{q}{2q-2}}+\slnorm{\nabla^3 \bmd}{2}^{\frac{q-2}{2q-2}}+1),
		\end{aligned}
	\end{equation}
	which together with \eqref{rhodotu} and \eqref{low-ord-est} leads to
	\begin{align*}
		\int_{0}^{T}(\|\tdiv\bmu\|_{L^\infty}+\|G\|_{L^\infty})^{1+\frac{1}{q}}dt&\leq C\int_{0}^{T}\left(\slnorm{\rho \dot{\bmu}}{q}^{\frac{q}{2q-2}\left(1+\frac{1}{q}\right)}+\slnorm{\nabla^3 \bmd}{2}^{\frac{(q-2)(q+1)}{(2q-2)q}}+1\right)dt\leq C(T),
	\end{align*}
	since $\frac{(q-2)(q+1)}{(2q-2)q}<2$ for any $q>2$.
\end{proof}
\begin{lemma}
		Assume that \eqref{TH1. 1-1} holds. Then for $q>2$, there exists a constant $C$ depending on $T, q, \mu, \beta, \gamma$,   $\|\bmu_0\|_{H^1}$, $\|\nabla\bmd_0\|_{H^1}$ and $\|\rho_0\|_{W^{1,q}}$ such that
	\begin{equation}\label{higerderiv}
		\begin{aligned}
			\sup_{0 \le t \le T}&(\slnorm{\nabla \rho}{q}+t^\frac{1}{2}\slnorm{\nabla^2 \bmu}{2}+t^\frac{1}{2}\slnorm{\nabla \bmd_t}{2}+t^\frac{1}{2}\slnorm{\nabla^3 \bmd}{2})+\int_{0}^{T}\|\nabla^2\bmu\|_{L^q}^{1+\frac{1}{q}}+\|\nabla^3\bmd\|_{L^q}^{1+\frac{1}{q}}dt\\
			&+\int_{0}^{T}t\|\bmu_t\|_{H^1}^2+t \slnorm{\nabla^2\bmd_t }{2}^2+t\|\nabla^2\bmu\|_{L^q}^2+t\|\nabla^3\bmd\|_{L^q}^2dt\le C(T). 
		\end{aligned}
	\end{equation}
\end{lemma}
	\begin{proof}
		It follows from $\eqref{CEL-sys}_1$ that 
		\begin{equation*}\label{hdproof_1}
			\begin{aligned}
				(\nabla \rho)_t+(\bmu \cdot \nabla) \nabla \rho +\nabla \bmu \cdot \nabla \rho +\nabla \rho \tdiv\bmu+\rho \nabla \tdiv \bmu=0.  
			\end{aligned}
		\end{equation*}
		For $q>2$, multiplying the above equation by $q |\nabla \rho|^{q-2} \nabla \rho$, and integrating the resultant equation over $\Omega$, one gets from integration by parts that
		\begin{equation*}\label{hdproof_2}
			\begin{aligned}
				\frac{d}{dt} \int |\nabla \rho |^q dx 
				&\le C\int|\nabla \rho|^q |\nabla \bmu|dx +C\int |\nabla \rho|^{q-1} |\nabla \tdiv\bmu|dx\\
				&\le C \slnorm{\nabla \bmu}{\infty} \slnorm{\nabla \rho}{q}^q+C\slnorm{\nabla \tdiv\bmu}{q}\slnorm{\nabla \rho}{q}^{q-1}, 
			\end{aligned}
		\end{equation*}
		which implies that
		\begin{equation}\label{hdproof_3}
			\begin{aligned}
				\frac{d}{dt} \slnorm{\nabla \rho}{q} &\le C \slnorm{\nabla \bmu}{\infty} \slnorm{\nabla \rho}{q}+C\slnorm{\nabla \tdiv\bmu}{q}. 
			\end{aligned}
		\end{equation}
		Next, we get from the standard $ L^p$-estimates for the elliptic system \eqref{ellipest_1} and \eqref{nbdG}, \eqref{rhodotuLq} that
		\begin{equation}\label{hdproof_7}
			\begin{aligned}
				\slnorm{\nabla^2 \bmu}{q} &\le C\slnorm{\nabla \text{div}\bmu}{q}\le C\left\|\nabla\left(\frac{G+P-\bar{P}}{2\mu+\lambda}\right)\right\|_{L^q}\\
				&\le C(\slnorm{\nabla G}{q}+\slnorm{\nabla \rho}{q}+\slnorm{G}{\infty}\slnorm{\nabla\rho}{q})\\
				&\le C(\|\rho\dot{\bmu}\|_{L^q}+\slnorm{\nabla^3 \bmd}{2}^{1-\frac{2}{q}} +1) +C\left(\slnorm{\rho \dot{\bmu}}{q}^{\frac{q}{2q-2}}+\slnorm{\nabla^3 \bmd}{2}^{\frac{q-2}{2q-2}}+1\right) \slnorm{\nabla \rho}{q}\\
				&\le C \slnorm{\nabla \rho}{q}^{\frac{2q-2}{q-2}}+C\|\rho\dot{\bmu}\|_{L^q}+C\slnorm{\nabla^3 \bmd}{2}^{\frac{q-2}{q}}+C, 
			\end{aligned}
		\end{equation}
		which implies that
		\begin{equation}\label{hdproof_11}
				\log(e+\slnorm{\nabla^2 \bmu}{q}) \le C \log (e+\|\rho\dot{\bmu}\|_{L^q}+\slnorm{\nabla^3 \bmd}{2}+\slnorm{\nabla \rho}{q}). 
		\end{equation}
		Thus, by using \eqref{lemma2-5}, \eqref{hdproof_11} and \eqref{hdproof_4}, we have
		\begin{equation}\label{nbduLinf}
			\begin{aligned}
				\slnorm{\nabla \bmu}{\infty} &\leq C\left(\slnorm{\tdiv\bmu}{\infty} \log(e+\slnorm{\nabla^2\bmu}{q})+\slnorm{\nabla \bmu}{2}+1\right)\\
				&\leq C(\slnorm{\rho \dot{\bmu}}{q}^{\frac{q}{2q-2}}+\slnorm{\nabla^3 \bmd}{2}^{\frac{q-2}{2q-2}}+1)\log (e+\|\rho\dot{\bmu}\|_{L^q}+\slnorm{\nabla^3 \bmd}{2}+\slnorm{\nabla \rho}{q})\\
				&\leq C(\slnorm{\rho \dot{\bmu}}{q}^{\frac{q}{2q-2}}+\slnorm{\nabla^3 \bmd}{2}^{\frac{q-2}{2q-2}}+1)\log [(e+\|\rho\dot{\bmu}\|_{L^q}+\slnorm{\nabla^3 \bmd}{2})(e+\slnorm{\nabla \rho}{q})]\\
				&\leq C(1+\slnorm{\rho \dot{\bmu}}{q}+\slnorm{\nabla^3 \bmd}{2})\log(e+\|\nabla\rho\|_{L^q}).
			\end{aligned}
		\end{equation}
		From \eqref{hdproof_7}, we also have 
		\begin{equation}\label{hdproof_12}
			\slnorm{\nabla \text{div} \bmu}{q}\le C(1+\|\nabla G\|_{L^q}+\|G\|_{L^\infty})(e+\slnorm{\nabla\rho}{q}). 
		\end{equation}
		Substituting \eqref{nbduLinf} and \eqref{hdproof_12} into \eqref{hdproof_3} and using Gronwall's inequality leads to
		\begin{equation}
			\sup_{0 \le t \le T}\slnorm{\nabla \rho(t)}{q} \le C(T),
		\end{equation}
	which together with \eqref{hdproof_7} and \eqref{rhodotuLq-est} shows that
	\begin{equation}
		\int_{0}^{T}\|\nabla^2\bmu\|_{L^q}^{1+\frac{1}{q}}+t\|\nabla^2\bmu\|_{L^q}^2 dt\leq C(T).
	\end{equation}
Moreover, similar to \eqref{hdproof_7}, one has
\begin{equation}
	\begin{aligned}
		\slnorm{\nabla^2 \bmu}{2} &\le C\slnorm{\nabla \text{div}\bmu}{2}\le C\left\|\nabla\left(\frac{G+P-\bar{P}}{2\mu+\lambda}\right)\right\|_{L^2}\\
		&\leq C(\|\nabla\rho\|_{L^2}+\|\nabla G\|_{L^2}+\|G\|_{L^\frac{2q}{q-2}}\|\nabla\rho\|_{L^q})\\
		&\leq C(T)(\|\nabla G\|_{L^2}+1)\\
		&\leq C(T)(\|\rho\dot{\bmu}\|_{L^2}+\|\nabla\bmd\|_{L^\infty}^2\|\nabla^2\bmd\|_{L^2}),
	\end{aligned}
\end{equation}
which combined with \eqref{est_rhodotu} leads to
\begin{equation}\label{nabula2_u_control}
	\sup_{0\le t \le T}t\|\nabla^2\bmu\|_{L^2}^2\leq C(T).
\end{equation}
Combing \eqref{nabula2_u_control} with \eqref{Poincare-in} and \eqref{low-ord-est} yields that
	\begin{equation}
		\begin{aligned}
		\int_{0}^{T} t \|\bmu_t\|_{H^1}^2dt &\leq C\int_{0}^{T}t (\|\sqrt{\rho}\bmu_t\|_{L^2}^2+\|\nabla \bmu_t\|_{L^2}^2)dt\\
		&\leq C\int_{0}^{T} t (\|\sqrt{\rho}\dot{\bmu}\|_{L^2}^2+\|\bmu\|_{L^\infty}^2\|\nabla\bmu\|_{L^2}^2+\|\nabla \dot{\bmu}\|_{L^2}^2+\|\nabla\bmu\|_{L^4}^4+\|\bmu\|_{L^\infty}^2\|\nabla^2\bmu\|_{L^2}^2)dt\\
		&\leq C\int_{0}^{T} t \|\sqrt{\rho}\dot{\bmu}\|_{L^2}^2+t \|\nabla \dot{\bmu}\|_{L^2}^2+t \|\nabla^2\bmu\|_{L^2}^2dt\leq C(T).
		\end{aligned}
	\end{equation}
For the time-weighted estimates of $\nabla^3\bmd$ and $\nabla\bmd_t$, it is obvious from \eqref{est_rhodotu} and \eqref{rho_upperbdd} that
\begin{equation}\label{tnbd2dt}
	\sup_{0\le t \le T}(t\slnorm{\nabla \bmd_t}{2}^2+t\|\nabla^3\bmd\|_{L^2}^2)+ \int_0^T t \slnorm{\nabla^2\bmd_t }{2}^2 dt \le C(T).
\end{equation}
Using standard elliptic estimates \eqref{elliptic_est2} and \eqref{low-ord-est} lead to
\begin{align*}
		\|\nabla^3\bmd\|_{L^q}&\leq C(\|\bmd_t\|_{W^{1,q}}+\|\bmu\cdot\nabla\bmd\|_{W^{1,q}}+\||\nabla\bmd|^2\bmd\|_{W^{1,q}})\\
		&\leq C\|\bmd_t\|_{W^{1,q}}+C(\|\bmu\|_{L^\infty}+\|\nabla\bmd\|_{L^\infty})\|\nabla\bmd\|_{W^{1,q}}\\
		&\quad+C(\|\nabla\bmu\|_{L^q}\|\nabla\bmd\|_{L^\infty}+\|\nabla\bmd\|_{L^q}\|\nabla\bmd\|_{L^\infty}^2)\\
		&\leq C\|\bmd_t\|_{W^{1,q}}+C\|\nabla\bmd\|_{H^2}+C\|\nabla\bmu\|_{H^1},
\end{align*}
which together with \eqref{lemma2-3}, \eqref{low-ord-est}, \eqref{nbdt-est} and \eqref{tnbd2dt} yields that
\begin{equation}
	\begin{aligned}
		\int_{0}^{T}\|\nabla^3\bmd\|_{L^q}^{1+\frac{1}{q}}&\leq C\int_{0}^{T}\|\nabla\bmd_t\|_{L^q}^{1+\frac{1}{q}}dt+C\int_{0}^{T}\|\bmd_t\|_{H^1}^2+\|\nabla\bmd\|_{H^2}^2+\|\nabla\bmu\|_{H^1}^2dt\\
		&\leq C\int_{0}^{T}\left(\|\nabla\bmd_t\|_{L^2}^\frac{2}{q}\|\nabla^2\bmd_t\|_{L^2}^{1-\frac{2}{q}}\right)^{1+\frac{1}{q}}dt+C\\
		&\leq C\int_{0}^{T}t\|\nabla^2\bmd_t\|_{L^2}^2dt+C\int_{0}^{T}t^{-1+\frac{2}{q^2+q+2}}+C\leq C,
	\end{aligned}
\end{equation}
and 
\begin{equation}
	\int_{0}^{T}t\|\nabla^3\bmd\|_{L^q}^2\leq C\int_{0}^{T}t(\|\bmd_t\|_{H^2}^2+\|\nabla\bmd\|_{H^2}^2+\|\nabla\bmu\|_{H^1}^2)dt\leq C.
\end{equation}
Therefore, the proof of lemma is completed.
\end{proof}

\section{Proofs of Theorem \ref{main-thm}}

\subsection{Global existence of strong solutions}

With all the \textit{a-priori} estimate at hand, we can prove global existence and uniqueness of strong solution in Theorem \ref{main-thm} by a standard argument as in \cite{HuangXDLiJ2016, Zhongxin2026}.  We sketch the proof as follows. First, for $\delta \ge 0$, we construct
\begin{equation*}
    \begin{aligned}
        &\rho_0^{\delta}=j_{\delta}*\rho_0 +\delta \ge \delta > 0\\
        &\bmu_0^{\delta}=j_{\delta}*\bmu_0\\
        &\bmd_0^{\delta}=j_{\delta}*\bmd_0\\
    \end{aligned}
\end{equation*}
where $j_{\delta}$ is the standard mollifying kernel of width $\delta$. Hence we have $(\rho_0^{\delta}, \bmu_0^{\delta}, \bmd_0^{\delta}) \in H^{\infty}$, and 
\begin{equation*}
    \begin{aligned}
        \lim_{\delta \rightarrow{0}} (\Vert \rho_0^{\delta}-\rho_0 \Vert_{W^{1, q}} +(\Vert \bmu_0^{\delta}-\bmu_0 \Vert_{H^{1}} +(\Vert  \nabla \bmd_0^{\delta}-\nabla \bmd_0 \Vert_{H^{1}})=0, 
    \end{aligned}
\end{equation*}
According to Lemma \ref{lwp} , there exist a $\tilde{T} >0$ such that the simplified compressible \textit{Ericksen-Leslie} system \eqref{CEL-sys}-\eqref{BC} has a unique local strong solution $(\rho^{\delta}, \bmu^{\delta}, \bmd^{\delta})$ on $[0, \tilde{T}]$ with initial data $(\rho_0^{\delta}, \bmu_0^{\delta}, \bmd_0^{\delta})$. And we define 
\begin{equation*}
    \begin{aligned}
        T^{*}=\sup \left\{T| \sup_{0\le t \le T}\hnorm{(\rho^{\delta}, \bmu^{\delta}, \nabla \bmd^{\delta})}{2}{\Omega} < \infty \right\}. 
    \end{aligned}
\end{equation*}
and define up to a time  $T_0$ that  solution $(\rho^{\delta}, \bmu^{\delta}, \bmd^{\delta}) $ satisfying the condition needed in  Lemma \ref{lwp}. 
Due to \eqref{est_AB}, \eqref{higerderiv}, we have $T^{*} = T_0$. 
If $T^{*} < \infty$,  for all $0 < T<T^*$, we have 
$$\sup_{0\le t\le T} \hnorm{\rho^{\delta}}{2}{\Omega} \le C^* $$ 
by \eqref{higerderiv} and standard argument in \cite{HuangXDLiJ2016}. This combine with Lemma \ref{lwp} contradicts the definition of $T^*$. Thus we have $$T^*= \infty.$$
So $(\rho^{\delta}, \bmu^{\delta}, \bmd^{\delta})$ exists on $\Omega \times [0, \infty)$ and satisfying all estimates in \eqref{TH1. 1-4}. Then letting $\delta \rightarrow{0} $ we proof that$(\rho, \bmu, \bmd) $ is global solution of system  \eqref{CEL-sys}-\eqref{symmetric assumption}. 

\subsection{Uniform upper bound of density and large time behaviors of strong solutions}

	The uniform estimates of density is a direct consequence of proposition \ref{time-independent upperbdd_rho} and compactness. It only remains to prove the large time behaviors \eqref{large-time behavior}. 
	
	First, we can follow the similar arguments in \cite{HuangXDLiJ2016} to derive
	\begin{equation}\label{rho_converge}
		\begin{aligned}
			\lim_{t \rightarrow +\infty} \slnorm{\rho-\bar{\rho}_0}{p}^2 (t) =0, 
		\end{aligned}
	\end{equation}
	for any $p \in [1, \infty)$. In fact,
	it follows from \eqref{G-def}, \eqref{B-E-estimate},  \eqref{est_AB} , \eqref{est_G_H1}, and \eqref{rho_upperbdd} that
	\begin{equation*}
		\begin{aligned}
			\int_{0}^{T} \slnorm{P(\rho)-\bar{P}(\rho)}{2}^2 dt \leq C \int_{0}^{T} \|G\|_{L^2}^2 + \slnorm{\sqrt{2\mu+\lambda} \tdiv\bmu}{2}^2 dt \leq C, 
		\end{aligned}
	\end{equation*}
with $C$ being independent of $T$. Thus, we have
	\begin{equation}\label{time-integrate-sup}
			\int_1^{\infty} \slnorm{P(\rho)-\bar{P}(\rho)}{2}^2 +\slnorm{G}{2}^2 dt \leq C 
	\end{equation}
	 Similar arguments as in \cite{HuangXDLiJ2016} lead to
	 \begin{equation}
	 	\sup_{N\leq t \leq N+1}\slnorm{P(\rho)-\bar{P}(\rho)}{2}^2(t)\leq C\int_{N}^{N+1}\left(\slnorm{P(\rho)-\bar{P}(\rho)}{2}^2+\|G\|_{L^2}^2\right)dt,
	 \end{equation}
	 which holds for any $N\in \mathbb{N}$. Letting $N\rightarrow+\infty$, this combined with \eqref{time-integrate-sup} yields \eqref{rho_converge}.

	Next, in view of \eqref{est_d1_proof_1}, one gets from \eqref{low-ord-est} that
	\begin{equation*}
		\begin{aligned}
			\frac{d}{dt}\slnorm{\nabla \bmd(t)}{4}^4&\leq C\|\nabla\bmd(t)\|_{L^\infty}^2\|\nabla^2\bmd(t)\|_{L^2}^2+C(\|\bmu(t)\|_{L^\infty}^2+\|\nabla\bmd(t)\|_{L^\infty}^2)\|\nabla\bmd(t)\|_{L^2}^2\|\nabla\bmd(t)\|_{L^\infty}^2\\
			&\leq C(\|\nabla^2\bmd(t)\|_{L^2}^2+\|\nabla\bmd(t)\|_{L^\infty}^2)
		\end{aligned}
	\end{equation*}
	which yields that, for any $s,t\in[N,N+1]$,
	\begin{equation}
		\slnorm{\nabla \bmd(t)}{4}^4-\slnorm{\nabla \bmd(s)}{4}^4\leq C\int_{N}^{N+1}(\|\nabla^2\bmd(t)\|_{L^2}^2+\|\nabla\bmd(t)\|_{L^\infty}^2)dt
	\end{equation}
	Integrating again with respect to $s$ over $(N,N+1)$ and using \eqref{low-ord-est}, \eqref{rho_upperbdd} yield that
	\begin{align*}
		\sup_{N\leq t\leq N+1}\slnorm{\nabla \bmd(t)}{4}^4&\leq C\int_{N}^{N+1}\slnorm{\nabla \bmd(s)}{4}^4ds+C\int_{N}^{N+1}(\|\nabla^2\bmd(t)\|_{L^2}^2+\|\nabla\bmd(t)\|_{L^\infty}^2)dt\\
		&\leq C\int_{N}^{N+1}(\|\nabla^2\bmd(t)\|_{L^2}^2+\|\nabla\bmd(t)\|_{L^\infty}^2)dt
	\end{align*} 
	From \eqref{d-est2}, one gets 
	\begin{equation*}
		\int_{1}^{\infty}(\|\nabla^2\bmd(t)\|_{L^2}^2+\|\nabla\bmd(t)\|_{L^\infty}^2)dt\leq C
	\end{equation*}
	Letting $N\rightarrow+\infty$ yields that
	\begin{equation}\label{L4-nbdd-conv}
		\lim\limits_{t\rightarrow+\infty }\|\nabla\bmd(t)\|_{L^4}^4=0.
	\end{equation}
	
	Finally, it follows from \eqref{estAB_proof_6} , \eqref{est_AB_another} and  \eqref{rho_upperbdd} that, given any $T>0$,
	\begin{equation}\label{A_I_0_ineq}
			\frac{1}{2}\frac{d}{dt}A^2(t)\le \frac{d}{dt}I_0 + C(\slnorm{\sqrt{2\mu+\lambda} \tdiv\bmu}{2}^2+\|\nabla^2\bmd\|_{L^2}^2)
	\end{equation}
	holds for any $t\in(0,T)$ with $C$ independent of $T$. Using the defination of $I_0$ in \eqref{estAB_proof_1} and \eqref{idd} lead to
	\begin{equation} \label{I_0_t-A_t}
		\begin{aligned}
			|I_0(t)|& \leq C \slnorm{\nabla \bmd(t)}{4}^2 \slnorm{\tdiv \bmu(t)}{2}\\
			&\leq \frac{\mu}{4}\slnorm{\tdiv \bmu(t)}{2}^2 +C\slnorm{\nabla \bmd(t)}{4}^4\\
			&\leq\frac{1}{4}A^2(t)+ C\slnorm{P-\bar{P}}{2}(t)+C\slnorm{\nabla \bmd(t)}{4}^4, 
		\end{aligned}
	\end{equation}
	Then, for any $s, t\in [N, N+1]$ with $N \in \mathbb{N}$ and $s<t$, we can obtain from integrating  \eqref{A_I_0_ineq} respect to $t$ over $(s, t)$ and using \eqref{I_0_t-A_t} that
	\begin{equation*}
		\begin{aligned}
			\frac{1}{2}A^2(t) &\leq \frac{1}{2}A^2(s) + I_0(t)-I_0(s) + C\int_{N}^{N+1} (\slnorm{\sqrt{2\mu+\lambda} \tdiv\bmu}{2}^2+\|\nabla^2\bmd\|_{L^2}^2)dt \\
			&\leq \frac{1}{4}A^2(t)+ C\slnorm{P-\bar{P}}{2}(t)+C\slnorm{\nabla \bmd(t)}{4}^4+\frac{1}{2}A^2(s) +C\slnorm{\nabla \bmd}{\infty}^2(s)\\
			&\quad+ C\int_{N}^{N+1}(\slnorm{\sqrt{2\mu+\lambda} \tdiv\bmu}{2}^2+\|\nabla^2\bmd\|_{L^2}^2) dt,
		\end{aligned}
	\end{equation*}
	where we have used
	\begin{align*}
		|I_0(s)|\leq \|\nabla\bmd(s)\|_{L^2}^2\|\tdiv\bmu(s)\|_{L^2}\|\nabla\bmd(s)\|_{L^\infty}^2\leq C\|\nabla\bmd(s)\|_{L^\infty}^2
	\end{align*}
due to \eqref{B-E-estimate}, \eqref{rho_upperbdd} and the last inequality in \eqref{I_0-est}. Integrating respect to $s$ over $(N, N+1)$ yields that
\begin{equation*}
	\begin{aligned}
		\frac{1}{4}A^2(t)  &\leq C\slnorm{P-\bar{P}}{2}(t)+C\|\nabla\bmd(t)\|_{L^4}^4\\
		&\quad+ C \int_{N}^{N+1} A^2(s) +\slnorm{\nabla \bmd}{\infty}^2(s) + \slnorm{\sqrt{2\mu+\lambda} \tdiv\bmu}{2}^2(s)+\|\nabla^2\bmd\|_{L^2}^2(s) ds. 
	\end{aligned}
\end{equation*}
Letting $N\rightarrow+\infty$ and using \eqref{rho_converge}, \eqref{L4-nbdd-conv}, \eqref{A-def}, \eqref{B-E-estimate} and \eqref{d-est2} gives that
\begin{equation*}
		\lim_{t \rightarrow +\infty} A^2(t) =0
\end{equation*}
which leads to
\begin{equation}\label{nbdu-conv}
		\lim_{t \rightarrow +\infty}\slnorm{\nabla \bmu}{2}^2\leq \lim_{t \rightarrow +\infty}(\slnorm{G}{2}^2 +\slnorm{P(\rho)-\bar{P}(\rho)}{2}^2  )=0, 
\end{equation}
and
\begin{equation}\label{2nbdd-conv}
	\lim_{t \rightarrow +\infty}\slnorm{\nabla^2 \bmd}{2}^2 \leq \lim_{t \rightarrow +\infty} A^2(t) =0. 
\end{equation}
Note from \eqref{est_nbdup} and \eqref{rho_upperbdd} that, for any $t\geq 1$
\begin{equation*}
	\|\nabla\bmu(t)\|_{L^p}+\|\nabla^2\bmd\|_{L^p}\leq C,
\end{equation*}
where $C$ is independent of $t$. 
Therefore, for any $p>2$, we get form H\"{o}lder's inequality, \eqref{nbdu-conv} and \eqref{2nbdd-conv} that
\begin{equation*}
		\lim_{t \rightarrow +\infty}\left(\slnorm{\nabla \bmu}{p}+\slnorm{\nabla^2 \bmd}{p}\right) \leq \lim_{t \rightarrow + \infty}\left(\slnorm{\nabla \bmu}{2}^{\frac{1}{p}}\slnorm{\nabla \bmu}{2(p-1)}^{\frac{p-1}{p}}+\slnorm{\nabla^2 \bmd}{2}^{\frac{1}{p}}\slnorm{\nabla^2 \bmd}{2(p-1)}^{\frac{p-1}{p}}\right)=0,
\end{equation*}
which complete the proof. 

\bigskip

\noindent {\bf Acknowledgments:}
 The research of Yu Mei is supported by the National Natural Science Foundation of China No. 12571241 and No. 12371227.	

	\bibliographystyle{plain}
	\bibliography{ref}

\end{document}